\documentclass[a4paper]{article} 
\usepackage{amssymb,mathtools,amsthm,tikz-cd,tkz-berge}
\tikzset{VertexStyle/.style ={shape= circle, fill = white, inner sep = 0pt, outer sep = 0pt, minimum size = 4pt,draw}}
\tikzset{EdgeStyle/.style={>=stealth,->}}
\tikzset{R0/.style={dotted}}
\tikzset{R1/.style={}}
\tikzset{CO/.style={dotted}}
\newcounter{k}\newcounter{M}
\newcounter{kk}
\addtocounter{kk}{-1}
\newcommand*\charac\chi
\DeclareMathOperator{\Star}{Star}
\DeclareMathOperator{\Stab}{Stab}
\DeclareMathOperator{\Id}{Id}
\DeclareMathOperator{\der}{der}

\DeclareMathOperator{\rank}{rk}
\DeclareMathOperator{\lcm}{lcm}
\DeclareMathOperator{\cayl}{Cay}
\DeclareMathOperator{\schrei}{Sch}

\DeclareMathOperator{\Aut}{Aut}
\DeclareMathOperator{\DL}{DL}
\newcommand*{\acts}{.}
\newcommand*{\lamp}{\mathcal L}
\newcommand*{\A}{\mathcal A}
\newcommand*{\defi}[1]{\emph{#1}}
\newcommand*{\tn}[1]{\textnormal{#1}} 
\newcommand*{\defegal}{\coloneqq}
\newcommand*{\iso}{\simeq}
\newcommand*{\field}[1]{\mathbf{#1}}
\newcommand*{\Z}{\field{Z}}
\newcommand*{\N}{\field{N}}
\newcommand*{\barN}{\overline{\N}}
\newcommand*{\Nstar}{\N_0}
\newcommand*{\hyphen}{\nobreakdash-\hspace{0pt}}
\newcommand*{\setst}[2]{\{#1 \mid #2\}}
\newcommand*{\presentation}[2]{\langle#1\mid#2\rangle}
\DeclarePairedDelimiter\abs{\lvert}{\rvert}
\newcommand*{\Bruijn}{\mathcal{B}}
\newcommand*{\sw}{\mathcal{S}}
\newcommand*{\orB}{\vec\Bruijn}
\newcommand*{\unB}{\Bruijn}
\newcommand*{\orSchrei}[3]{\vec\schrei(#1,#2,#3)}
\newcommand*{\unSchrei}[3]{\schrei(#1,#2,#3)}
\newcommand*{\orCayley}[2]{\vec\cayl(#1,#2)}
\newcommand*{\unCayley}[2]{{\cayl(#1,#2)}}
\newcommand*{\orkSw}[3]{\vec\sw_{#1,#2,#3}}
\newcommand*{\unkSw}[3]{{\sw_{#1,#2,#3}}}
\newcommand*{\orSw}[3]{\vec\sw_{#2,#3}}
\newcommand*{\unSw}[3]{{\sw_{#2,#3}}}
\usepackage{xcolor}
\newtheorem{proposition}{Proposition}[subsection]
\newtheorem{theorem}{Theorem}[subsection]
\newtheorem{lemma}{Lemma}[subsection]
\newtheorem{corollary}{Corollary}[subsection]
\newtheorem{proposition2}{Proposition}[section]
\newtheorem{theorem2}{Theorem}[section]
\newtheorem{lemma2}{Lemma}[section]

\theoremstyle{definition} 
\newtheorem{definition}{Definition}[subsection]
\newtheorem{notation}{Notation}[subsection]
\newtheorem{remark}{Remark}[subsection]
\newtheorem{definition2}{Definition}[section]
\newtheorem{remark2}{Remark}[section]
\newcommand*{\titre}{
Lamplighter groups, de Bruijn graphs, spider-web graphs and their spectra}
\title{\titre}
\author{R. Grigorchuk\thanks{The authors were supported by the Swiss National Science Foundation. R.G. was also supported by NSF  grant  DMS-1207669}, P.-H. Leemann and T. Nagnibeda}

\date{\today}
\providecommand\keywords[1]{\textbf{Keywords:} #1}
\begin{document}
\maketitle
%
%
%
%
%
\begin{abstract}
We study the infinite family of spider-web graphs $\{\unkSw{k}{N}{M}\}$, $k\geq 2$, $N\geq 0$ and $M\geq 1$, initiated in the 50-s in the context of network theory.
It was later shown in physical literature that these graphs have remarkable percolation and spectral properties.
We provide a mathematical explanation of these properties by putting the spider-web graphs in the context of group theory and algebraic graph theory.
Namely, we realize them as tensor products of the well-known de Bruijn graphs $\{\unB_{k,N}\}$ with cyclic graphs $\{ C_M\}$ and show that these graphs are described by the action of the lamplighter group $\lamp_k = \Z/k\Z\wr \Z$ on the infinite binary tree.
Our main result is the identification of  the infinite limit of $\{\unkSw{k}{N}{M}\}$, as $N,M \rightarrow \infty$, with the Cayley graph of the lamplighter group $\lamp_k$ which, in turn, is one of the famous Diestel-Leader graphs $DL_{k,k}$.
As an application we compute the spectra of all spider-web graphs and show their convergence to the discrete spectral distribution associated with the Laplacian on the lamplighter group.
\end{abstract}

%
%
%
%
%
\section{Introduction}

Exchange of methods and ideas between physics and mathematics has a long history and led to many spectacular results.
Spectral theory is one of many examples of such a fruitful interaction.

The goal of this paper is to show how group theory can be successfully applied to explain and study some objects of interest in physics and in the network theory, with emphasis on their spectral properties. 

Spider-web networks were introduced by Ikeno in 1959 \cite{Ikeno} in order to study systems of telephone exchanges.
They were later shown to enjoy interesting properties in percolation, see \cite{MR1099797}, \cite{MR1176870} and \cite{MR2257252}.
Our work stems from the paper~\cite{MR2904386} by Balram and Dhar where they are interested in the asymptotic properties of the sequence of spider-web graphs $\{\unkSw{k}{N}{M}\}$, for $k=2$.
In particular, they find, using an interesting approach based on symmetries, the spectra of graphs $\unkSw{2}{N}{M}$ and observe that they converge to a discrete limiting distribution as $M,N \rightarrow \infty$.

Here, we develop a method that leads to the full understanding of this infinite discrete model, including its spectral characteristics, via finite approximations, using the notion of Benjamini-Schramm limit of graphs that has lately become very important in probability theory.
A remarkable feature of the model that we discover is that it is related to one of the most interesting and important test-cases in combinatorial group theory, both algebraically and from the spectral and probabilistic viewpoints, the lamplighter groups.

It is interesting to observe that the study of more and more models in pure and applied mathematics, theoretical and statistical physics and in computer science see abstract groups appearing naturally, not only describing symmetries, but also serving as non-commutative time scales in dynamical systems, describing monodromies, providing automatic structure etc.
However, the work was mostly done on the cubic lattices and the Bethe lattice and its close relatives (as e.g. the modular group of the surface groups), in relation to percolation, the Ising model, the sandpile model and many more.
It is therefore particularly interesting that, as we show, a lattice with very different geometry, the Diestel-Leader graph associated to the lamplighter group, arises naturally as the limit of the spider-web graphs.
The explicit identification, presented here, of this infinite lattice as the limit of spider-web graphs, leads immediately to spectral results, but also potentially to future advancements in the study of percolation on these graphs, as explained in  \cite{MR2257252}.

The aim of the present paper is therefore to provide a unified rigorous framework for studying spider-web graphs $\unkSw{2}{N}{M}$, for any $k\geq 2$, and their spectra, and to identify their limit, as $M,N\rightarrow\infty$,  as a particular Cayley graph of the lamplighter group $\lamp_k=\Z/k\Z\wr \Z$ (see Subsection~\ref{subsectionLamplighter} for the definition) known (see \cite{MR2138121}) as the Diestel-Leader graph .

Convergence of spider-web graphs to this graph comes from the following structural result that we prove.
For any $k\geq 2$, the oriented spider-web graph $\orkSw{k}{N}{M}$ decomposes into the tensor product of the graph $\orkSw{k}{N}{1}$ and the oriented cycle $\vec C_M$ of length $M$.
It is then useful to note that the sequence $\orkSw{k}{N}{1}$ is nothing else than the well-studied sequence of de Bruijn graphs, see Subsection~\ref{subsectionBruijn}.
De Bruijn graphs are famous for their useful connectivity properties and, being both Hamiltonian and Eulerian, are used both in mathematics, where they represent word overlaps in symbolic dynamical systems, and in applications, as for example for the discrete model for the Bernoulli map or for genome assembly in bioinformatics \cite{CPT}.
Our results imply that, for each $k\geq 2$, the two-parameter family $\{\unkSw{k}{N}{M}\}$ of spider-web graphs is in fact a natural extension of the family of de Bruijn graphs $\unB_{k,N}$.

We then prove a result of independent interest, that de Bruijn graphs are isomorphic to another well-known sequence of finite graphs provided by a self-similar action of the lamplighter group $\lamp_k$ by automorphisms on the $k$-regular rooted tree, see \cite{MR1866850}. 
Our main result then follows: the sequence of spider-web graphs $\orkSw{k}{N}{M}$ (respectively $\unkSw{k}{N}{M}$) converges, as $M,N \rightarrow \infty$ to the Cayley graph of the lamplighter group $\lamp_k$, see Theorem~\ref{ThmConvergenceM1} (respectively Corollary~\ref{CorConvergenceUnoriented}).
There is also an alternative more direct way to prove that de Bruijn graphs $\unB_{k,N}$ (as well as the spider-web graphs) converge to the Diestel-Leader graph $\DL(k,k)$ \cite{PH}.

The spectra of de Bruijn graphs have been computed by Delorme and Tillich in \cite{MR1621013}.
We extend this computation to all spider-web graphs by using their tensor product structure.
The spectral approximation in the context of Benjamini-Schramm limits (see Definition~\ref{DefBSConv})  then ensures that the spectra of finite spider-web graphs converge to the spectral distribution corresponding to the limit graph.
As mentioned above, this spectral distribution coincides with one of those associated with the lamplighter group.


The spectral theory of discrete Laplacians on lattices and on Cayley graphs is a very popular topic related to the theory of random walks on groups initiated by Kesten, Atiyah's theory of $L^2$-invariants, Kadison-Kaplansky Conjecture and many more.
It can be viewed as a discrete analogue of the famous Kac's question ``Can one hear the shape of a drum''.
The lamplighter group is a very interesting object from the viewpoint of spectral theory. It was open for a longtime whether the Laplacian spectrum on a Cayley graph can have a discrete component.
This was answered in \cite{MR1866850} where it was shown that the spectrum of a certain Cayley graph of the lamplighter group is pure point. 

On the other hand, it follows from \cite{MR1952401} by Elek that for the ``standard'' generating set (the one that corresponds to the algebraic structure of the lamplighter group), the spectrum contains no eigenvalue.
This is illustrated on Figure~\ref{FigSPectre}, where the left column  corresponds to the Diestel-Leader graph $\DL(2,2)$ which, as we have already mentioned, is isomorphic to a specific Cayley graph of $\lamp_2$, see \eqref{presentation2} on page~\pageref{presentation2}, whereas the right column corresponds to the Cayley graph of $\lamp_2$ with respect to the standard generating set (see \eqref{presentation1} on page~\pageref{presentation1}).

This is the first example of a dramatic change that the Laplacian spectrum can undergo under local perturbations, even in the presence of a large underlying group of symmetries.
Recently other examples of this type were discovered by the first and the third authors in collaboration with Lenz \cite{GLN1, GLN2}, in the context of group actions with aperiodic order. 

The first two lines of Figure~\ref{FigSPectre} are the histograms of the spectral measure (respectively for linear and logarithmic $y$-axes) and the last line shows the corresponding density functions.
In both cases the graphics correspond to approximations of the infinite graph by graphs with $2^N$ vertices (provided by the action of $\lamp_2$ on the infinite full binary tree, see Subsection~\ref{SubSecActionLk}).
On the left, the Diestel-Leader graph is approximated by de Bruijn (and equivalently spider-web $\unkSw{2}{N}{M}$ for any $M$) graphs (see Remark \ref{RmqIndepM}); in this case the exact spectral measure is known (\cite{MR1866850}, see also \eqref{SpectralMeasureInfinite} 
on page~\pageref{SpectralMeasureInfinite}).
It is not known for the Cayley graph of $\lamp_k$ with respect to the standard generators.
\begin{figure}[htp]
\centering
\includegraphics[scale=0.64]{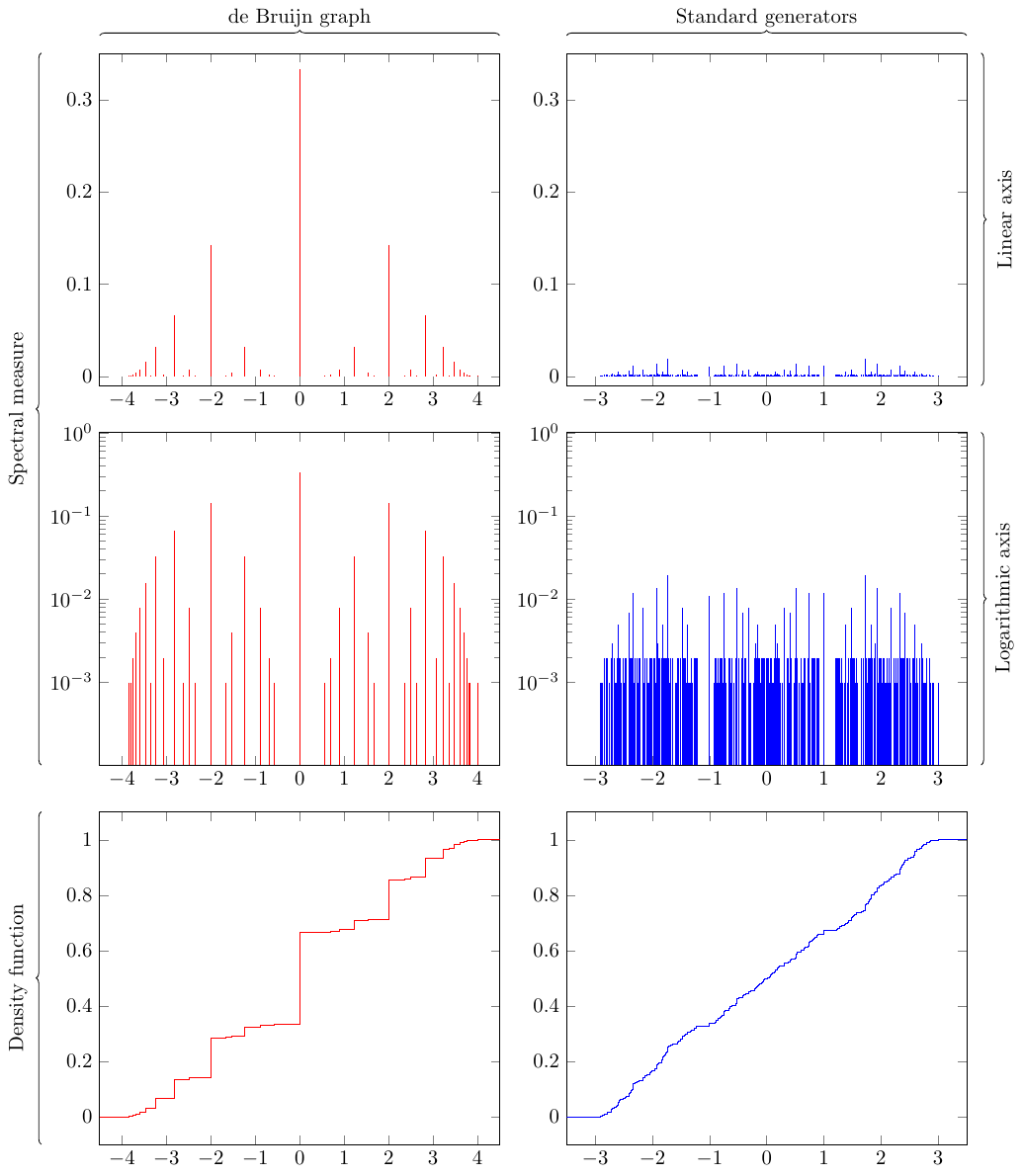}
\caption{
Approximations of the spectral measure and of the density function of two Cayley graphs of the lamplighter group $\lamp_2$.
The standard one on the right (see Subsection~\ref{subsectionLamplighter}), and the Diestel-Leader graph $\DL(2,2)$ approximated by de Bruijn (and spider-web $\unkSw{2}{N}{M}$, any $M$) graphs on the left.
The histograms correspond to graphs with $2^N$, $N=10$, vertices.
}
\label{FigSPectre}
\end{figure}

The paper is organized as follows.
Section~\ref{sectionConv} introduces all the relevant notions from graph theory and contains some useful preliminary results.
In particular, we recall the notion of the topological space of marked graphs.
We then turn our attention to the study of tensor product of graphs in Section~\ref{sectionTensor}.
In Subsection~\ref{SubSTensorConv} we investigate how the tensor product behaves under the convergence in the space of marked graphs.
Starting from Subsection~\ref{SubSTensorSchrei} we specialize to the case of graphs defined by a group action (so called Schreier graphs, see definition~\ref{DefCayley}), and further to the case of the lamplighter group in Subsection~\ref{subsectionLamplighter}.
The structure of spider-web graphs is analyzed in Section~\ref{sectionSpider}.
Subsection~\ref{sectionSpiderLamplighter} establishes in particular a connection between spider-web graphs and lamplighter groups.
Section~\ref{sectionSpectra} contains spectral computations on spider-web graphs.
In the last Section~\ref{sectionGeneralResults} we provide some further results about spider-web graphs and their relation to lamplighters. 
It turns out that all  $\orkSw{k}{N}{M}$ are Schreier graphs of the lamplighter group $\lamp_k$ (Theorem~\ref{ThmWeakIsomorphism}) and we identify the subgroups to which they correspond (Theorems~\ref{ThmWeakIsomorphism} and~\ref{ThmSwSchreier}).
It is then shown in Theorem~\ref{thmTransitivity} that for all $k$, the graph $\orkSw{k}{N}{M}$ is transitive if and only if $M\geq N$.
In Theorem~\ref{thmSpiderasfiniteCayley}, we show that if moreover $N$ divides $M$, it is a Cayley graph of a finite quotient of the lamplighter group $\lamp_k$.

The authors would like to thank Vadim Kaimanovich for his interest in this work and for inspiring discussions on the subject of the paper.
They also would like to thank the anonymous referee for his careful reading and his valuable remarks.
%
%
%
%
\section{Definitions and preliminaries}\label{sectionConv}
In this paper we deal with both oriented and non-oriented graphs and allow loops and multiple edges.
It will be convenient for us to work with the definition of a graph suggested by Serre \cite{MR0476875}.
\begin{definition2}\label{DefGraph}
A \defi{(non-oriented) graph} $\Gamma=(V,E)$ consists of two disjoint sets $V$ (\defi{vertices}) and $E$ (\defi{oriented edges}), and three functions $\iota,\tau\colon E\to V$ (initial vertex and end vertex) and $\bar{\phantom{e}}\colon E\to E$ (the inverse edge) satisfying $\iota(\bar e)=\tau(e)$, $\bar{\bar{ e}}=e$ and $\bar e\neq e$.
A \defi{non-oriented edge} is a pair $\{e,\bar e\}$.

An \defi{oriented graph} $\vec\Gamma=(V,\vec E)$ is given by a set of vertices $V$, a set of oriented edges $E$ and two functions $\iota,\tau\colon E\to V$ with no conditions on them.
To avoid confusion, from now on we will always write \defi{graph} for non-oriented graph and \defi{oriented graph} otherwise.

An \defi{orientation} $\mathcal O$ on a graph $\Gamma$ is the choice of an edge in each of the pairs $\{e,\bar e\}$.
For each choice of an orientation $\mathcal O$ on $\Gamma=(V,E)$, we define the oriented graph $\vec\Gamma=(V,\vec E)$ where $\vec E=\mathcal O$ and $\iota$ and $\tau$ are restrictions on $\vec E$ of the original functions.

The \defi{underlying graph} of an oriented graph $\vec\Gamma=(V,\vec E)$ is the graph $\underline{\vec\Gamma}(V,E)$, with $E\defegal\vec E\sqcup\setst{\bar e}{e\in \vec E}$, where $\bar e$ is the formal inverse of $e$.
For $\bar e$, we define $\iota(\bar e)\defegal\tau(e)$, $\tau(\bar e)\defegal\iota(e)$ and $\bar{\bar e}\defegal e$.
\end{definition2}
The operations of choosing an orientation on a graph and of taking the underlying graph of an oriented graph are mutually inverse in the following sense.
Given a graph $\Gamma$, the underlying graph of the oriented graph obtained by choosing an orientation on $\Gamma$, is $\Gamma$ itself.
On the other hand, given an oriented graph $\vec\Gamma$ there exists an orientation on the underlying graph such that the resulting oriented graph is $\vec\Gamma$ itself.

\begin{remark2}
In this paper we will only consider connectedness in the weak (non-oriented) sense. In particular, a connected component of $\vec\Gamma$ is a connected component of $\underline{\vec\Gamma}$ with the orientation coming from $\vec\Gamma$.
\end{remark2}

Let $\Gamma=(V,E)$ be a graph, oriented or not.
The \defi{in-degree}, respectively the \defi{out-degree}, of a vertex $v$ is the number of edges $e$ with initial vertex $v$, respectively end vertex $v$.
If the graph is non-oriented, then both notions coincide and are simply called \defi{degree}.
The graph $\Gamma$ is said to be \defi{locally finite} if every vertex has both finite in-degree and finite out-degree.
%
%
Note that if $e$ is a loop in a graph, it contributes $1$ to the in-degree, but its inverse edge $\bar e$ also contributes $1$.
Therefore, the non-oriented loop $\{e,\bar e\}$ contributes $2$ to the degree, since $\bar e\neq e$.

The \defi{adjacency matrix} of graph $\Gamma$ is the symmetric matrix $A_\Gamma=(a_{ij})_{i,j\in V}$ with $a_{ij}$ the number of edges from $i$ to $j$.
For an oriented graph $\vec \Gamma$ the adjacency matrix is not necessary symmetric, but we have: $A_{\underline{\vec\Gamma}}=A_{\vec\Gamma}+A_{\vec\Gamma}^T$.

A \defi{morphism of oriented graphs} $\vec\Gamma_1\to\vec\Gamma_2$ is a function $\phi\colon V_1\cup \vec E_1\to V_2\cup \vec E_2$ such that $\phi(V_1)\subseteq V_2$, $\phi(\vec E_1)\subseteq \vec E_2$ and for every edge $e$ in $\Gamma_1$, we have $\iota\bigl(\phi(e)\bigr)=\phi\bigl(\iota(e)\bigr)$ and $\tau\bigl(\phi(e)\bigr)=\phi\bigl(\tau(e)\bigr)$.
A \defi{morphism of graphs} is defined in the same way, with the additional requirement that $\overline{\phi(e)}=\phi(\bar e)$.
Let \(\Delta\) be a graph (oriented or not) and let \(v\) any vertex of \(\Delta\).
The \defi{star} of \(v\) is the set \(\setst{e\in E}{\iota(e)=v}\).
Remark that any morphism \(\phi\colon\Delta_1\to \Delta_2\) induces, for any vertex \(v\) of \(\Delta_1\), a~map: $\phi_v\colon\Star_v\to\Star_{\phi(v)}$.
A morphism \(\phi\colon\Delta_1\to\Delta_2\) is a \defi{covering} if all the induced maps \(\phi_v\) are bijections.
In this case, we say that $\Delta_1$ \defi{covers} $\Delta_2$.

Let $\Gamma$ be a  graph. A \defi{path} $p$ in $\Gamma$ from $v$ to $w$ is an ordered sequence of edges $(e_1,e_2,\dots,e_n)$ such that $\iota(e_1)=v$, $\tau(e_n)=w$ and for all $1\leq i<n$ we have $\tau(e_i)=\iota(e_{i+1})$.
The inverse of the path $p=(e_1,e_2,\dots,e_n)$ is the path $\bar p=(\bar e_n,\dots,\bar e_1)$.
The \defi{length} of a path $(e_1,e_2,\dots,e_n)$ is equal to $n$.
A path is said to be \defi{reduced} if it does not contain subsequences of the form $e\bar e$.
\begin{definition2}\label{defDerangement}
Let $\Gamma$ be a non-oriented graph, $p=(e_1,\dots, e_n)$ a path of length $n$ in $\Gamma$ and $\mathcal O$ an orientation on $\Gamma$.
The \defi{signature $\sigma(p)$ of $p$ with respect to $\mathcal O$} is an ordered sequence of $\pm1$ of length $n$, where there is a $1$ in the position $i$ if and only if $e_i$ belongs to $\mathcal O$ and a $-1$ otherwise. 

The \defi{derangement of $p$ with respect to $\mathcal O$}, $\der(p)$, is the sum of the $\pm 1$ in the signature of $p$. The derangement of a path of length $0$ is $0$.
It follows from the definition that $\der(\bar p)=-\der(p)$ and that $\sigma(\bar p)$ is the sequence $-\sigma(p)$ readed backward.

The \defi{derangement of $\Gamma$ with respect to $\mathcal O$} is 
\[
\der(\Gamma)\defegal\min\setst{\abs{\der(p)}}{p \tn{ is a closed path in }\Gamma\tn{ and }\der(p)\neq 0},
\]
where this minimum is defined to be $0$ if there is no closed path in $\Gamma$ with non-zero derangement.
\end{definition2}
We also need a variant of this definition for an oriented graph $\vec\Gamma=(V,\vec E)$ and $p$ a path in the underlying graph.
The \emph{signature of $p$}, respectively the \emph{derangement of $p$}, are the signature, respectively the derangement, of $p$ with respect to the orientation coming from $\vec\Gamma$.
The \emph{derangement} of $\vec\Gamma$ is $\der(\vec\Gamma)\defegal\der(\underline{\vec\Gamma})$, for the orientation on $\underline{\vec\Gamma}$ coming from $\vec\Gamma$.
\begin{definition2}\label{DefTopo}
A \defi{marked graph} is a couple $(\Gamma,v)$ where $\Gamma$ is a graph and $v$ a vertex of $\Gamma$, called the \defi{root} of the marked graph.
For an (oriented) marked graph $(\Gamma,v)$ we will denote by $(\Gamma,v)^0$ the connected component containing~$v$.

We denote $\mathcal G_*$ (respectively $\vec{\mathcal G}_*$) the set of connected marked (respectively connected oriented marked) graphs, up to isomorphisms of marked graphs.
\end{definition2}

The set $\mathcal G_*$ (respectively $\vec{\mathcal G}_*$) can be topologized by considering for example the following distance:
$d\bigl((\Gamma,v),(\Delta,w)\bigr)=\frac{1}{1+r}$, where $r$ is the biggest integer such that the ball of radius $r$ centered at $v$ in $\Gamma$ and the ball of the same radius centered at $w$ in $\Delta$ are isomorphic as marked (respectively marked oriented) graphs.
If the two graphs are isomorphic as marked graphs, then the distance is defined to be $0$.
For an oriented marked graph $(\vec \Gamma,v)$, the ball of radius $r$ centered at $v$ is the oriented subgraph of $\vec\Gamma$ such that its underlying graph is the ball of radius $r$ centered at $v$ in $\underline{\vec\Gamma}$.
For any integer $d$, the subspaces $\mathcal G_{*,\leq d}$ of $\mathcal G_*$ and $\vec{\mathcal G}_{*,\leq d}$ of $\vec{\mathcal G}_*$ consisting of graphs with both maximal in-degree and out-degree bounded by $d$ are compact.

It is easy to check that, if $(\vec\Gamma,v)$ and $(\vec\Delta,w)$ are two oriented marked graphs, then 
\[
d_{\vec{\mathcal G_*}}\bigl((\vec\Gamma,v),(\vec\Delta,w)\bigr)\geq d_{\mathcal G_*}\bigl((\underline{\vec\Gamma},v),(\underline{\vec\Delta},w)\bigr).
\]
It immediately implies the following proposition.
\begin{proposition2}\label{ConvergenceOrientedVsNonOriented}
If a sequence of oriented marked graphs $(\vec\Gamma_n,v_n)$ converges to $(\vec\Gamma,v)$, then the sequence $(\underline{\vec\Gamma_n},v_n)$ converges to $(\underline{\vec\Gamma},v)$.
\end{proposition2}

Since $\mathcal G_{*,\leq d}$ is a metric space which is separable, compact and complete, by Prokhorov's Theorem \cite{MR0084896} the space of Borel probability measures on it is compact in the weak topology.
There is a natural way to attach a Borel probability measure to a finite graph $\Gamma$: by choosing the root uniformly at random.
More formally, the measure associated to $\Gamma$ is 
$
\frac 1{\abs V}\sum_{v\in V}\delta_{(\Gamma,v)^0}
$, where $\delta$ is a Dirac measure.

\begin{definition2}[\cite{MR1873300}]\label{DefBSConv}
Let $\Gamma_n$ be a sequence of finite graphs and let $\lambda_{\Gamma_n}$ be the Borel probability measures associated.
We say that $\Gamma_n$ is \defi{Benjamini-Schramm convergent} with limit $\lambda$ if $\lambda_{\Gamma_n}$ converges to $\lambda$ in the weak topology in the space of Borel probability measures on $\mathcal G_*$.

In the particular case where $\lambda$ is a Dirac measure concentrated on one transitive graph $\Gamma$, we say that $\Gamma_n$ converges to $\Gamma$ in the sense of Benjamini-Schramm.
\end{definition2}
The same definitions hold in $\vec{\mathcal G_*}$.
In this paper we will deal with Schreier graphs coming from group actions, so we also need to establish a similar setup for labeled graphs.
\begin{definition2}
An \defi{oriented labeled graph} is a triple $\bigl(\vec\Gamma,X,l\bigr)$, where $\vec\Gamma=(V,\vec E)$ is an oriented graph, $X$ an alphabet (the set of labels) and $l\colon\vec E\to X$ a function (the labeling).
The \defi{underlying labeled graph} is $(\underline{\vec\Gamma},X,l')$ where $l'\colon E \to X\sqcup X^{-1}$ such that for every edge $e$ in $\vec E$ we have $l'(e)\defegal l(e)$ and $l'(\bar e)\defegal l(e)^{-1}$.
A morphism of (oriented) labeled graphs over the same alphabet which preserves the labeling is called a \defi{strong morphism}.
If we forget about the labeling and the morphism is only between (oriented) graphs, we say the this is a \defi{weak morphism}.
\end{definition2}
Typical examples of labeled graphs are Cayley graphs and more generally Schreier graphs, see Definition~\ref{DefCayley}.

Every concept that can be expressed using morphisms in the category of (oriented) graphs has an obvious ``strong'' analog in the category of (oriented) labeled graph with strong morphisms.
Thus, we have strong isomorphisms, strong coverings, a distance in the space of marked labeled graphs and hence a notion of strong convergence and of strong Benjamini-Schramm convergence.
%
%
%
%
\section{Tensor product of graphs}\label{sectionTensor}
\begin{definition2}
Let $\Gamma=(V,E)$ and $\Delta=(W,F)$ be two (oriented) graphs.
Their \defi{tensor product} is the (oriented) graph $\Gamma\otimes\Delta$, with vertex set $V\times W$, where there is an edge $(e,f)$ from $(v_1,w_1)$ to $(v_2,w_2)$ if $e$ is an edge from $v_1$ to $v_2$ in $\vec\Gamma$ and $f$ is an edge from $w_1$ to $w_2$ in $\vec\Delta$.
If $\Gamma=(V,E)$ and $\Delta=(W,F)$ are non-oriented graphs, then the inverse of the edge $(e,f)$ is the edge $(\bar e,\bar f)$.

If $\Gamma=(V,E)$ has labeling $l\colon E\to X$ and $\Delta=(W,F)$ has labeling $l'\colon F\to Y$, the tensor product has labeling $l\times l'\colon E\times F\to X\times Y$.
\end{definition2}
For two oriented graphs $\vec\Delta$ and $\vec\Gamma$ we have $\vec\Gamma\otimes\vec\Delta\iso\vec\Delta\otimes\vec\Gamma$ and $\vec\Delta\otimes\emptyset\iso\emptyset$, where $\emptyset$ denotes the empty graph. 

The tensor product of (oriented) graphs is the categorical product in the category of (oriented) graphs.
This implies that for any pair of morphisms $\phi\colon\vec\Gamma\to\vec\Delta$ and $\phi'\colon\vec\Gamma'\to\vec\Delta'$, $\phi\otimes\phi'$ is a morphism from $\vec\Gamma\otimes\vec\Gamma'$ to $\vec\Delta\otimes\vec\Delta'$ and that $\phi\otimes\phi'$ is an isomorphism if and only if $\phi$ and $\phi'$ are isomorphisms.
\begin{lemma2}\label{lemmaCovering}
For $i=1,2$, let $\phi_i\colon\Gamma_i\to\Delta_i$ be a covering.
Then, $\phi_1\otimes\phi_2\colon \Gamma_1\otimes\Gamma_2\to\Delta_1\otimes\Delta_2$ is a covering.
The same result is true for oriented graphs.
\end{lemma2}
\begin{proof}
Let $(v_1,v_2)$ be any vertex in $\Gamma_1\otimes\Gamma_2$.
Since the $\phi_i$'s are coverings, the induced morphisms $(\phi_i)_{v_i}\colon\Star_{v_i}\to\Star_{\phi_i(v_i)}$ are bijections.
On the other hand, by definition of the tensor product, there is a natural bijection between $\Star_{v_1}\times\Star_{v_2}$ and $\Star_{(v_1,v_2)}$.
Under this bijection, the map $(\phi_1\otimes\phi_2)_{(v_1,v_2)}$ corresponds to $(\phi_1)_{v_1}\times (\phi_2)_{v_2}$ and is therefore a bijection.
\end{proof}
\begin{definition2}\label{defLine}
Let $\vec\Gamma=(V,\vec E)$ be an oriented graph.
The \defi{line graph} of $\vec\Gamma$ is the oriented graph $L(\vec\Gamma)$ with vertex set $\vec E$ (the edge set of $\vec\Gamma$) and with an edge from $e$ to $f$ if we have $\tau(e)=\iota(f)$ (that is $f$ ``directly follows'' $e$) in~$\vec\Gamma$ .
\end{definition2}

\begin{lemma2}\label{PropLineTensor}
For $i=1,2$, let $\vec\Gamma_i=(V_i,\vec E_i)$ be an oriented graph.
Then the graphs $L(\vec\Gamma_1)\otimes L(\vec \Gamma_2)$ and $L(\vec\Gamma_1\otimes \vec \Gamma_2)$ are isomorphic.
\end{lemma2}
\begin{proof}
Vertices of $L(\vec\Gamma_1\otimes \vec \Gamma_2)$ are in $1$-to-$1$ correspondance with edges of $\vec\Gamma_1\otimes \vec \Gamma_2$ and therefore in $1$-to-$1$ correspondance with pairs of edges in $\vec E_1\times \vec E_2$.
On the other hand, vertices of $L(\vec\Gamma_1)\otimes L(\vec \Gamma_2)$ are in $1$-to-$1$ correspondance with $\{(v_1,v_2)\, |\, v_i \textnormal{ a vertex in } L(\vec\Gamma_i)\}$.
Therefore, vertices of $L(\vec\Gamma_1)\otimes L(\vec \Gamma_2)$ are also in $1$-to-$1$ correspondance with pairs of edges in $\vec E_1\times \vec E_2$.

Now, in $L(\vec\Gamma_1\otimes \vec \Gamma_2)$ there is an edge from $(e_1,e_2)$ to $(f_1,f_2)$ if and only if, for $i=1,2$, $f_i$ directly follows $e_i$ in $\Gamma_i$.
The same relation holds in $L(\vec\Gamma_1)\otimes L(\vec \Gamma_2)$, which proves the isomorphism.
\end{proof}

%
%
%
\subsection{Tensor product and convergence}\label{SubSTensorConv}
Recall that for a marked labeled graph $(\vec\Gamma, v)$, we denote by $(\vec\Gamma, v)^0$ the connected component of $\underline{\vec\Gamma}$ containing the root, with the orientation coming from $\vec\Gamma$.
\begin{theorem}\label{ThmLimit}
If $(\vec\Gamma_n,v_n)$ converges (in $\vec{\mathcal G}_*$) to $(\vec\Gamma,v)$ and $(\vec\Theta_m,y_m)$ converges to $(\vec\Theta,y)$ then the following diagram is commutative
\center
\includegraphics{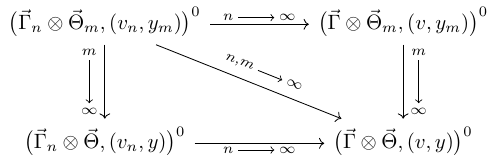}
\end{theorem}
\begin{proof}
Take any $\epsilon>0$.
By convergence, there exists $n_0$ and $m_0$ such that for every $n\geq n_0$ the graphs $(\vec\Gamma_n,v_n)$ and $(\vec\Gamma,v)$ are at distance lesser than $\epsilon$ and such that for every $m\geq m_0$ the graphs $(\vec\Theta_m,v_m)$ and $(\vec\Theta,v)$ are too at distance lesser than $\epsilon$.

Let $(\vec\Delta_1,v)$, $(\vec\Delta_2,w)$, $(\vec\Pi_1,x)$ and $(\vec\Pi_2,y)$ be four elements of $\vec{\mathcal G}_*$.
We affirm that the distance between 
$\bigl(\vec\Delta_1\otimes\vec\Pi_1,(v,x)\bigr)^0$ and $\bigl(\vec\Delta_2\otimes\vec\Pi_2,(w,y)\bigr)^0$ is lesser or equal to the maximum of $d\bigl((\vec\Delta_1,v),(\vec\Delta_2,w)\bigr)$ and $d\bigl((\vec\Pi_1,x),(\vec\Pi_2,y)\bigr)$.
Lemma~\ref{lemmaPath} below implies in turn that $\bigl(\vec\Gamma_n\otimes\vec\Theta_m,(v_n,y_m)\bigr)^0$ and $\bigl(\vec\Gamma\otimes\vec\Theta,(v,y)\bigr)^0$ are at distance less than $\epsilon$, which proves the convergence when both $n$ and $m$ grow together.

Now, if we take first the limit on $n$ we can use this result with $\vec\Theta_m$ constant to find
\[\lim_{n\to \infty}\left(\bigl(\vec\Gamma_n\otimes\vec\Theta_m,(v_n,y_m)\bigr)^0\right)=\bigl(\vec\Gamma\otimes\vec\Theta_m,(v,y_m)\bigr)^0.\]
Taking then the limit on $m$ (with $\vec\Gamma$ constant) we have that the upper right triangle is commutative.
A similar argument proves the commutativity of the downer left triangle.
\end{proof}
Note that Theorem~\ref{ThmLimit} holds also for non-oriented marked graphs as well as for labeled marked graphs with strong morphisms.

We will now prove the technical result used in the proof of Theorem~\ref{ThmLimit}.
\begin{lemma}\label{lemmaPath}
Let $\vec\Gamma$ and $\vec\Delta$ be two oriented graphs and $p$ be a path in $\underline{\vec\Gamma\otimes\vec\Delta}$ from $(x,v)$ to $(y,w)$.
Then there exists paths $q$ in $\underline{\vec\Gamma}$ from $x$ to $y$ and $r$ in $\underline{\vec\Delta}$ from $v$ to $w$ with same signature as $p$.

More precisely, given a non-negative integer $n$ and a sequence $\sigma$ of $\pm 1$ of length $n$, there is a bijection between the set of paths $p$ from $(x,v)$ to $(y,w)$ in $\underline{\vec\Gamma\otimes\vec\Delta}$ of signature $\sigma$ and the set of couples $(q,r)$ where $q$ is a path in $\underline{\vec\Gamma}$ from $x$ to $y$ and $r$ a path in $\underline{\vec\Delta}$ from $v$ to $w$, both of signature $\sigma$.
\end{lemma}
\begin{proof}
It is obvious that the second statement implies the first one.
%
By definition of the tensor product, we have a 
function $\phi$ from the set of paths from $(x,v)$ to $(y,w)$ to the set of couples $(q,r)$ where $q$ is a path in $\underline{\vec\Gamma}$ from $x$ to $y$ and $r$ a path in $\underline{\vec\Delta}$ from $v$ to $w$.
Indeed, $\phi$ is the product of the left projection and the right projection.
This function naturally preserves the signature and is injective.
Now, if $q=(e_1,\dots,e_n)$ and $r=(e'_1,\dots,e'_n)$ have the same signature $\sigma$ then either $e_1$ belongs to $\vec\Gamma$ and $e'_1$ belongs to $\vec\Delta$, in which case we have an edge $(e_1,e_1')$ in $\vec\Gamma\otimes\vec\Delta$, or $\bar e_1$ belongs to $\vec\Gamma$ and $\bar e'_1$ belongs to $\vec\Delta$, in which case we have an edge $(\bar e_1,\bar e_1')$ in $\vec\Gamma\otimes\vec\Delta$.
By induction, it is possible to construct a path $p$ in $\underline{\vec\Gamma\otimes\vec\Delta}$ from $(x,v)$ to $(y,w)$ with signature $\sigma$.
\end{proof}
%
%
%
%
%
\subsection{Tensor product with an oriented cycle and the oriented line}
Let us first consider the special case when one of the factors in the tensor product is $\vec C_\infty$ or $\vec C_M$,
where $\vec C_\infty$ is the ``oriented line'' with $V_{\vec C_\infty}=\Z$ (the set of integers) and for each vertex $i$ there is a unique oriented edge from $i$ to $i+1$, and $\vec C_M$ is the ``oriented cycle of length $M$'': $V_{\vec C_M}=\Z/M\Z$ and for each $i$ there is a unique oriented edge from $i$ to $i+1$ modulo $M$.
Below, we will write $M\in\barN=\{1,2,\dots,\infty\}$ and $i\equiv j\pmod\infty$ will mean $i=j$.

In this subsection we will only consider oriented connected graphs $\vec\Gamma$.
Recall the notion of derangement of a path from Definition~\ref{defDerangement} that we will need here.
\begin{proposition}\label{PropConnCompIso}
For any oriented connected graph $\vec\Gamma$ and any $M\in\barN$, all connected components of $\vec\Gamma\otimes\vec C_M$ are isomorphic.
\end{proposition}
\begin{proof}
Fix a vertex $v$ of $\vec\Gamma$.
Since $\vec\Gamma$ is connected, for any vertex $w$ there is a path $q$ from $v$ to $w$ in $\underline{\vec\Gamma}$, with signature $\sigma(q)$.
For any integer $i$, there exists a path $r$ from $i$ to $i+\der(q) \pmod M$ in $\underline{\vec C_M}$ with signature $\sigma(q)=\sigma(r)$.
Therefore, there is a path $p$ in $\underline{\vec\Gamma\otimes\vec C_M}$ from $(v,i)$ to $\bigl(w,i+\der(q)\bigr)$.
Hence, for any vertex $(w,j)$ in $\underline{\vec\Gamma\otimes\vec C_M}$, there exists an integer $i$ such that $(w,j)$ is in the connected component of $(v,i)$.

On the other hand, since for any integers $i$ and $j$, the marked graphs $(\vec C_M,i)$ and $(\vec C_M,j)$ are isomorphic, say by an isomorphism $\phi_{i,j}$, we have connected components $\bigl(\vec\Gamma\otimes\vec C_M,(v,i)\bigr)^0$ and $\bigl(\vec\Gamma\otimes\vec C_M,(v,j)\bigr)^0$ are isomorphic by $\Id\otimes\phi_{i,j}$.
This implies that all connected components are isomorphic.
\end{proof}
\begin{theorem}\label{thmIsomorphism}
Let $\vec\Gamma$ be a connected locally finite oriented graph.
For any $M\in\barN$ and any vertex $v$ in  $\vec\Gamma$, the marked oriented graph $(\vec\Gamma,v)$ is isomorphic (as marked oriented graph) to $\bigl(\vec\Gamma\otimes\vec C_M,(v,0)\bigr)^0$ if and only if $\der(\vec\Gamma)\equiv0 \pmod M$.
\end{theorem}
\begin{proof}
Suppose that $\der(\vec\Gamma)\equiv 0\pmod M$.
For any vertex $w$ of $\vec\Gamma$ define $\rank(w)$, the \defi{rank} of $w$, to be the derangement of any path in $\underline{\vec\Gamma}$ from $v$ to $w$ taken modulo $M$.
This is well defined since for two such paths $p$ and $q$, the concatenated path $p\bar q$ is a closed path based at $v$ with derangement $0\pmod M$.
We define a morphism from $(\vec\Gamma,v)$ to $\bigl(\vec\Gamma\otimes\vec C_M,(v,0)\bigr)^0$ by $w\mapsto\bigl(w,\rank(w)\bigr)$ for vertices.
For the edges, it maps an edge $e$ from $w$ to $x$ to an edge from $\bigl(w,\rank(w)\bigr)$ to $\bigl(x,\rank(x)\bigr)$.
Note that the vertices $\bigl(w,\rank(w)\bigr)$ and $\bigl(x,\rank(x)\bigr)$ are indeed connected by an edge in the tensor product since $\rank(x)=\rank(w)+1$.
It is easy to see that this morphism is surjective and injective, and hence is an isomorphism.

Suppose now that $\der(\vec\Gamma)\not\equiv 0 \pmod M$.
This implies the existence of a closed path $q_0$ in $\underline{\vec\Gamma}$ from $v$ to $v$ with non-zero $\pmod M$ derangement and length $n$.
By the second part of Lemma~\ref{lemmaPath}, the set of closed paths $q$ based at $v$ and of length $n$ is in bijection with the set of (non necessarily closed) paths $p$ from $(v,0)$ to $\bigl(v,\der(q)\bigr)$, where we used the fact that for every signature $\sigma$, there is a unique path in $C_M$ with initial vertex $0$ and signature $\sigma$.
Hence, the number of closed paths in $\underline{\vec\Gamma\otimes\vec C_M}$ of length $n$ based at $(v,0)$ is at most the number of closed path of length $n$ based at $v$, minus one (namely the path $q_0$).
If $\vec\Gamma$ is locally finite (note that local finiteness of $\vec\Gamma$ is used only in this direction of the proof), there is only a finite number of such paths.
In this case, $(\vec\Gamma,v)$ and $\bigl(\vec\Gamma\otimes\vec C_M,(v,0)\bigr)^0$ cannot be isomorphic (as oriented marked graphs).
\end{proof}
\begin{remark}\label{LabelingOfCM}
Proposition~\ref{PropConnCompIso} and Theorem~\ref{thmIsomorphism} (and their proofs) are still true  in the category of labeled oriented graphs (with strong morphisms) if we identify the labeling $(l\times l')$ of the tensor product with its first coordinate $l$, which is the labeling of $\vec\Gamma$. In the following, we will always use this identification for tensor product of the form $\vec\Gamma\otimes\vec C_M$.
\end{remark}
We know by Proposition~\ref{PropConnCompIso} and Theorem~\ref{thmIsomorphism} that all connected components of $\vec\Gamma\otimes\vec C_M$ are isomorphic and we are able, in the locally finite case, to decide when they are isomorphic (as marked graphs) to $\vec\Gamma$.
To complete the description of $\vec\Gamma\otimes\vec C_M$ it remains to count the number of connected components.
This is the subject of the next proposition.
\begin{proposition}\label{PropNbrComponent}
For any connected oriented graph $\vec\Gamma$, and any $M\in\N$ and any $i\in\Z$, let $[i]$ denotes the unique representative of $i$ modulo $M$ such that $-M/2<[i]\leq M/2$.
For $M=\infty$, we define $[i]\defegal i$.
For any connected graph $\vec\Gamma$, the number of connected components of $\vec\Gamma\otimes\vec C_M$ is $M$ if and only if $\der(\vec\Gamma)\equiv 0\pmod M$.
Otherwise it is equal to the absolute value of $[\der(\vec\Gamma)]$.

In particular, the number of connected component of $\vec\Gamma\otimes\vec \Z$ is infinite if $\der(\vec\Gamma)=0$ and $\der(\vec\Gamma)$ otherwise.
\end{proposition}
\begin{proof}
Choose a vertex $v_0$ in $\vec\Gamma$.
For every vertex $w$ of $\vec\Gamma$ there is a path $q$ in $\underline{\vec\Gamma}$ from $w$ to $v_0$, of length $n$ and signature $\sigma$.
For any $i$ there is obviously a path $r$ in $\underline{\vec C_M}$ of length $n$ and signature $\sigma$ with initial vertex $[i]$ and end vertex $\bigl[i+\der(r)\bigr]$.
Hence, for every vertex $(w,[i])$ of the tensor product, there is a path $p$ from $(w,[i])$ to $\bigl(v_0,[i+\der(r)\bigr)]$ in $\underline{\vec\Gamma\otimes\vec C_M}$.
Therefore, to count the number of connected components of $\vec\Gamma\otimes\vec C_M$ it is sufficient to know when two vertices $(v_0,[i])$ and $(v_0,[k])$ are connected.
But they are connected if and only if $(v_0,[0])$ and $(v_0,[k-i])$ are connected.

Let $i_0$ be the non-zero integer with the smallest absolute value such that $(v_0,0)$ and $(v_0,[i_0])$ are connected by a path in $\underline{\vec\Gamma\otimes\vec C_M}$. If such an integer does not exist, put $i_0=M$.
The previous discussion implies that $i_0=M$ if and only if the number of connected components of $\vec\Gamma\otimes\vec C_M$ is $M$.
On the other hand, $i_0=M$ if and only if every path $p$ in $\underline{\vec\Gamma\otimes\vec C_M}$ with initial vertex $(v_0,0)$ and end vertex $(v_0,j)$ satisfies $j=0$, in which case $\der(p)\equiv 0 \pmod M$.
But this is equivalent (by Lemma~\ref{lemmaPath} and by the existence in $\underline{\vec C_M}$ of a path with arbitrary signature) to every closed path $q$ in $\underline{\vec\Gamma}$ with initial vertex $v_0$ having $\der(q)\equiv 0\pmod M$,
which is  equivalent to $\der(\vec\Gamma)\equiv0\pmod M$.

If $i_0\neq M$, we have either $M=\infty$ or $M/2<i_0\leq M/2$.
In both cases $i_0=[i_0]$.
For every integer $j$, since $(v_0,0)$ and $(v_0,i_0)$ are connected, their images $(v_0,[j])$ and $(v_0,\bigl[[j]+i_0\bigr])$ by the automorphism $\Id\otimes\phi_{0,[j]}$ are connected, where $\phi_{0,[j]}$ is the automorphism of $\vec C_M$ sending $i$ on $i+[j]$.
Hence the vertices $(v_0,\bigl[[j]-i_0\bigr])$ and $(v_0,\bigl[[j]-i_0+i_0\bigr])$ are also connected.
As a special case we have that $(v_0,0)$ and $(v_0,[-i_0])$ are connected.
Therefore we can suppose that $i_0$ is strictly positive and $0<i_0\leq M/2$.
We also have by induction that for all $j$, $(v_0,[j])$ is connected to $(v_0,k)$ for some $0\leq k\leq i_0$.
On the other hand, $(v_0,0)$, $(v_0,1)$, \dots{} and $(v_0,i_0-1)$ are in different connected components by minimality of $i_0$.
Hence, the number of connected components of $\vec\Gamma\otimes\vec C_M$ is $i_0$.

Let us now show that $i_0$ is equal to the absolute value of $[\der(\vec\Gamma)]$.
Take a path $q$ in $\underline{\vec\Gamma}$ with initial vertex $v_0$ and such that $\abs{\der(q)}=\der(\vec\Gamma)$.
By Lemma~\ref{lemmaPath} this gives a path $p$ in $\underline{\vec\Gamma\otimes \vec C_M}$ with $\der(p)=\der(\vec\Gamma)$, initial vertex $(v_0,0)$ and end vertex $(v_0,[\der(p)])$.
This implies (by minimality of $i_0$) that the absolute value of $[\der(\vec\Gamma)]$ is bigger or equal to $i_0$, which is the number of connected components of $\vec\Gamma$.

It remains to show that $i_0$ is bigger or equal to the absolute value of $[\der(\vec\Gamma)]$.
Now, if $(v_0,0)$ is connected by $p$ to $(v_0,i_0)$, the derangement of $p$ is equal to $i_0$ modulo $M$.
This gives us a closed path $q$ (from $v_0$ to $v_0$) in $\underline{\vec\Gamma}$ with derangement $i_0+aM$ for some integer $a$.
Since $[i_0+aM]=[i_0]=i_0$, we have found a path $q$ in $\vec\Gamma$ such that $[\der(q)]=i_0$.
On the other hand, we have $\der(\vec\Gamma)\geq\abs{\der(q)}$.
We still have to show that $\abs{\der(q)}\geq[\der(q)]$.
But the stronger inequality $\abs i\geq \abs{[i]}$ is true for every integer $i$.
Indeed, if $-M/2<i\leq M/2$ we have $i=[i]$ and therefore $\abs i= \abs{[i]}$. Otherwise, $\abs i >M/2\geq\abs{[i]}$.
\end{proof}
An analogous proposition holds for non-oriented graphs, where the derangement is replaced by the length of a path and minimum is replaced by greatest common divisor.
This gives a refinement of the following proposition: $\Gamma\otimes\Delta$ is connected if and only if $\Gamma$ and $\Delta$ are connected and at least one factor is non-bipartite (\cite{MR1788124}, Theorem 5.29).
%
%
%
%
%
\subsection{Tensor product of a Schreier graph and an oriented cycle}\label{SubSTensorSchrei}
Here we keep $\vec C_M$, $M\in \barN=\{1,2,\dots,\infty\}$,  as one factor of the tensor product and take the other one to be as follows.
\begin{definition}\label{DefCayley}
Let $G$ be a group with a finite generating set $X$.
The \defi{oriented (right) Cayley graph} $\orCayley{G}{X}$ is the oriented marked labeled graph with vertex set $G$ and with an oriented edge from $g$ to $h$ labeled $x$ if and only if $h=gx$, $x\in X$.
The standard choice for the root is $1_G$.

For $H\leq G$, a subgroup, we define the \defi{oriented (right) Schreier graph} $\orSchrei{G}{H}{X}$ to be the oriented marked labeled graph with vertex set $\setst{Hg}{g\in G}$ (the set of right $H$\hyphen cosets) and an edge with label $x$ from $Hg$ to $Hh$ if and only if $Hh=Hgx$.
Here the standard choice of the root is (the coset) $H$.

If $G$ acts on the right on a set $V$, we can define the \defi{graph of the action with respect to the generating set $X$} as the oriented labeled graph with vertex set $V$ and an edge from $v$ to $w$ labeled by $x$ for every generator $x\in X$ such that $v\acts x=w$.
\end{definition}
For every vertex $v$ in $V$, the connected component of the graph of the action with root $v$ is strongly isomorphic (as marked labeled oriented graph) to the Schreier graph $\orSchrei{G}{\Stab_G(v)}{X}$.

Observe that for any vertices $v$ and $w$ in $\vec\Gamma=\orCayley{G}{X}$, the oriented labeled marked graphs $(\vec\Gamma,v)$ and $(\vec\Gamma,w)$ are strongly isomorphic and thus $\orCayley{G}{X}$ is strongly vertex-transitive.
This is not correct for Schreier graphs.
Indeed, $\orSchrei{G}{H}{X}$ is in general not even weakly vertex-transitive.
\begin{remark}
For a generating set $X$ of $G$, we can look at its symmetrization $X^\pm\defegal\setst{x}{x\in X \tn{ or } x^{-1}\in X}$.
This is also a generating set of $G$.
If $e$ is the unique edge in $\orCayley{G}{X}$ with initial vertex $g\in G$ and label $x\in X^\pm$, define $\bar e$ to be the unique edge with initial vertex $gx$ and label $x^{-1}$.
It is easy to see that this operator $\bar{\phantom{e}}$ makes of the oriented graph $\orCayley{G}{X^\pm}$ a non-oriented graph, but with the possibility that $\bar e=e$.
We will note this graph $\unCayley{G}{X}$.
An important fact for us is that there is a strong isomorphism between $\unCayley{G}{X^\pm}$ and $\underline{\orCayley{G}{X}}$ if and only if there is no $x\in X$ such that $x^2=1$.
Moreover, in this case $\unCayley{G}{X^\pm}$ is a graph in the sense of Definition~\ref{DefGraph} (i.e. there is no $e$ such that $e=\bar e$).
Indeed, $x^2=1$ if and only if $x^{-1}=x$.
If $x^2=1$, then for every vertex $v$ in $\unCayley{G}{X}$, the edge with initial vertex $v$ and label $x$ is equal to the edge with initial vertex $v$ and label $x^{-1}$, but in $\underline{\orCayley{G}{X}}$ they are distinct by definition.
If there is no such $x$, the strong isomorphism is trivial.
The same observation also applies to Schreier graphs.
\end{remark}
\begin{definition}
Let $w$ be a word in the alphabet $X\sqcup X^{-1}$.
For $x\in X$, the \defi{exponent} of $x$ in $w$, $\exp_x(w)$ is the number of times $x$ appears in $w$ minus the number of times $x^{-1}$ appears in $w$.
We also define the \defi{exponent} of $X$ as the sum of exponents:
\[\exp_X(w)\defegal\sum_{x\in X}\exp_x(w).\]
\end{definition}
The definition immediately implies
\begin{lemma}
Let $G=\presentation{X}{\mathcal R}$ be a group presentation.
Then the derangement of a path in ${\orCayley{G}{X}}$ is exactly the exponent of its label.
\end{lemma}
\begin{proposition}\label{propPresentation}
Fix $M\in\barN$ and let $G=\presentation{X}{\mathcal R}$ be a group presentation such that $\exp_X(r)\equiv0 \pmod M$ for every relator $r\in \mathcal R$.
Then $\orCayley{G}{X}$ is strongly isomorphic to any connected component of $\orCayley{G}{X}\otimes\vec C_M$.
\end{proposition}
\begin{proof}
Let $p$ be a path with initial vertex $1$ in ${\unCayley{G}{X}}$ and let $w$ be its label.
Then $w=1$ in $G$ if and only if $p$ is closed.
But $w=1$ in $G$ if and only if $w=\prod h_ir_ih_i^{-1}$, where the $r_i$ are relators and the $h_i$ are words in $X\sqcup X^{-1}$.

On the other hand, by the previous lemma the derangement of $p$ is equal to 
\begin{align*}
\exp_X(w)&=\exp_X(\prod h_ir_ih_i^{-1})\\
&=\sum\bigl(\exp_X(h_i)+\exp_X(r_i)-\exp_X(h_i)\bigr)\\
&\equiv 0 \pmod M.
\end{align*}
We conclude using Theorem~\ref{thmIsomorphism} and Remark~\ref{LabelingOfCM}.
\end{proof}
\begin{lemma}\label{lemmaSchreierTensor}
Fix $M\in\barN$ and let $G=\presentation{X}{\mathcal R}$ be a group presentation such that $\exp_X(r)\equiv0 \pmod M$ for every relator $r\in \mathcal R$.
An oriented labeled graph $\vec\Gamma$ is the graph of an action of $G$ if and only if $\vec\Gamma\otimes\vec C_M$ is also the graph of an action of $G$.
\end{lemma}
\begin{proof}
Let $\vec\Theta$ be any $X$-labeled graph such that for each $x\in X$ and each vertex $v$, there is exactly one outgoing and one ingoing edge with label $x$.
It is clear that $\vec\Theta$ is (strongly isomorphic to) a graph of an action of $G=\presentation{X}{\mathcal R}$ if and only if for every $r\in \mathcal R$, and for every vertex $v$, the unique path with initial vertex $v$ and label $r$ is closed.

Now, fix $v$ a vertex in $\vec\Gamma$, $r$ a word on $X\sqcup X^{-1}$ and $0\leq i<M$. There is a unique path $p$ with initial vertex $v$ and label $r$ in $\vec\Gamma$ and a unique path $q$ with initial vertex $(v,i)$ and label $r$ in $\vec\Gamma\otimes\vec C_M$.
We have that $\tau(q)=\bigl(\tau(p),i+\der(p)\bigr)$. Therefore, if $r$ is a relator
we have $\tau(q)=(\tau(p),i)$ and
$p$ is closed if and only if $q$ is closed.
\end{proof}
Using this lemma and Proposition~\ref{PropConnCompIso} we have the following.
\begin{proposition}\label{propSchreierTensor}
Fix $M\in\barN$ and let $G=\presentation{X}{\mathcal R}$ be a group presentation such that $\exp_X(r)\equiv0 \pmod M$ for every relator $r\in \mathcal R$.
Let $H$ be a subgroup of $G$ and let $\vec\Gamma\defegal\orSchrei{G}{H}{X}$ be the corresponding Schreier graph.
Then, every connected component of $\vec\Gamma\otimes\vec C_M$ is the Schreier graph of $G$ with respect to $X$ and to the subgroup $H_M\defegal\setst{h\in H}{\exp_X(h)\equiv 0 \pmod M}$.
\end{proposition}
\begin{proof}
First, note that since $\exp_X(r)\equiv0 \pmod M$ for every relator $r$, the exponent of $g\in G$ is well defined modulo $M$.
By Proposition~\ref{PropConnCompIso} and Remark~\ref{LabelingOfCM}, all connected components of $\vec\Gamma\otimes\vec C_M$ are strongly isomorphic.
By the previous lemma, $\vec\Gamma\otimes\vec C_M$ is a graph of an action of $G$ and therefore all its connected components are Schreier graphs of $G$.

Now, let $v$ be a vertex in $\vec\Gamma$ corresponding to the subgroup $H$.
The subgroup $H_M$ consists of labels of paths from $(v,0)$ to $(v,0)$ in $\underline{\vec\Gamma\otimes\vec C_M}$.
By Lemma~\ref{lemmaPath}, for any signature $\sigma$, there is a bijection between the set of closed paths $p$ with initial vertex $(v,0)$ and signature $\sigma$ and the set of couples $(q,r)$ where $q$ is a closed path with initial vertex $v$, $r$ a closed path with initial vertex $0$, both of signature $\sigma$.
But there is a path $r$ from $0$ to $0$ with signature $\sigma$ in $\vec C_M$ if and only if $\der(r)\equiv 0\pmod M$, and in this case there is a unique such path.
Finally, we conclude using the fact that the labeling of $\vec\Gamma\otimes\vec C_M$ is inherited from the labeling of $\vec\Gamma$.
\end{proof}
Observe that if in Proposition~\ref{propSchreierTensor}, $\Gamma=\orCayley{G}{X}$ then it corresponds to a Schreier graph with $H=\{1\}$ and thus $H_M=\{1\}$ and every connected component of $\vec\Gamma\otimes\vec C_M$ is isomorphic to $\vec\Gamma$ itself.
\begin{remark}\label{rmkGeo}Lemma~\ref{lemmaSchreierTensor} and Proposition~\ref{propSchreierTensor} have a geometrical meaning.
If for every relator $r$, $\exp_X(r)\equiv0 \pmod M$, then $G$ naturally acts on $\vec C_M$ by $i\acts g\defegal i+\exp_X(g)$.
Therefore, the action of $G$ on $\vec \Gamma\otimes \vec C_M$ is the product of the actions.
In term of Schreier graphs, that exactly means that $H_M=H\cap L_M$, where $L_M$ is the subgroup of $G$ which stabilizes the vertex $i\in \vec C_M$.
\end{remark}
%
%
%
%
%
\subsection{The case of lamplighter groups}\label{subsectionLamplighter}
By the lamplighter group $\lamp_k$, for $k\geq 2$, we mean the restricted wreath product $\Z/k\Z\wr\Z\iso\bigr(\bigoplus_\Z\Z/k\Z\bigl)\rtimes\Z$ where $\Z$ acts on the normal subgroup $\A_k\defegal\bigoplus_\Z\Z/k\Z$ by shifting the coordinates.
It is easy to see that it is given by the presentation
\[
	\lamp_k\defegal\presentation{b,c}{c^k,[c,b^ncb^{-n}]; n\in \N},\tag{\dag}\label{presentation1}
\]
where $[x,y]=x^{-1}y^{-1}xy$ is the commutator of $x$ and $y$.
Observe that this in particular implies $[b^mcb^{-m},b^ncb^{-n}]=1$ in $\lamp_k$ for all $m$ and $n$ in $\Z$.

The subgroup $\A_k=\oplus_\Z\Z/k\Z$, called the abelian base of the wreath product, is generated by $\{b^icb^{-i}\}_{i\in\Z}$, while $\Z$ is generated by $b$.
The (right) action of $b$ on $\A_k$ is by shift, $b^{i}cb^{-i}\acts b=b^{i-1}cb^{-i+1}$.

Following \cite{MR1866850}, instead of the ``classical'' presentation \eqref{presentation1} of $\lamp_k$, in this paper we will use the following  presentation coming from the automaton presentation.
Consider the set $X_k=\{b,\bar c_1,\dots, \bar c_{k-1}\}$, where $\bar c_i=bc^i$.
Note that $c=b^{-1}\bar c_1$, so $X_k$ does generate $\lamp_k$.
It will be convenient to write $\bar c_0$ for $b=bc^0$, so that $X_k=\{\bar c_i\}_{i=0}^{k-1}$.
\[
	\lamp_k=\presentation{X_k}{(b^{-1}\bar c_1)^k,b(b^{-1}\bar c_1)^i{\bar c_i}^{-1},[b^{-1}\bar c_1,b^{n-1}\bar c_1b^{-n}];n\in \N, 2\leq i\leq k-1}.\tag{\ddag}\label{presentation2}
\]

It is possible to check that $1$ does not belong to $X_k$ and that $x\in X_k$ implies that $x^{-1}\notin X_k$.
In particular, the graph $\unCayley{\lamp_k}{X_k}$ is strongly isomorphic to $\underline{\orCayley{\lamp_k}{X_k}}$.
It is interesting to note that with this particular choice of generators, the graph $\unCayley{\lamp_k}{X_k}$ is weakly isomorphic to the Diestel-Leader graph $DL(k,k)$ (see \cite{MR2138121}).
\begin{remark}\label{RmkLamplighter}
It is easy to see that, if $G$ is any finite group and we consider the restricted wreath product $G\wr\Z$, where $\Z=\langle b\rangle$ and choose the generating set $\{bg\}_{g\in G}$, then the corresponding Cayley graph will be also weakly isomorphic to $\unCayley{\lamp_k}{X_k}$ and thus to the Diestel-Leader graph $\DL(k, k)$.
For the rest of the paper, we will focus on the lamplighter group $\lamp_k$.
\end{remark}
We immediately have
\[
	\exp_{X}(b^{-1}\bar c_1)^k=\exp_{X}b(b^{-1}\bar c_1)^i{\bar c_i}^{-1}=\exp_{X}([b^{-1}\bar c_1,b^{n-1}\bar c_1b^{-n}])=0.
\]
Hence, the presentation \eqref{presentation2} of $\lamp_k$ satisfies the hypothesis of Proposition~\ref{propPresentation} and we proved the following special case of Proposition~\ref{propSchreierTensor}.
\begin{proposition}\label{PropPresentation2}
For all $k\geq 2$ and $M\in\barN$, every connected component of $\orCayley{\lamp_k}{X_k}\otimes\vec C_M$ is strongly isomorphic to $\orCayley{\lamp_k}{X_k}$.
\end{proposition}
%
%
%
%
%
\section{Spider-web graphs and lamplighter groups}\label{sectionSpider}
A slightly different version of spider-web graphs, called spider-web networks, was first introduced by Ikeno in \cite{Ikeno}.
The $2$-parameter family $\{\unkSw{k}{N}{M}\}$ that we will presently define is a natural extension of the well-known $1$-parameter family of the de Bruijn graphs $\{\unB_{k,N}\}$, $k\geq 2$.
In \cite{MR2904386}, Balram and Dhar observed, in the special case $k=2$, some link between spider-web graphs and the Cayley graph of the lamplighter group $\lamp_2$.

The aim of this section is to discuss the definition of spider-web graphs $\unkSw{k}{N}{M}$ and to show that they converge to the Cayley graph of the lamplighter group $\lamp_k$.
This is our Theorem~\ref{ThmConvergenceM1} for the oriented case and Corollary~\ref{CorConvergenceUnoriented} for the non-oriented case.
In order to do that, we first prove Theorem~\ref{ThmWeakIsomorphism} which shows that de Bruijn graphs are weakly isomorphic to Schreier graphs of the lamplighter group.

From now on, we fix a $k\geq 2$ and omit to write it when it is not necessary. We will use the notations $\N=\{1,2,\dots\}$, $\Nstar=\N\cup\{0\}$ and $\barN=\N\cup\{\infty\}$.
%
%
%
%
%
\subsection{De Bruijn Graphs}\label{subsectionBruijn}
\begin{definition}\label{DefiDeBruijn}
For every $N\in\Nstar$, the $N$\hyphen dimensional \defi{de Bruijn graph} $\orB_{k,N}$ on $k$ symbols is the oriented labeled graph with vertex set $\{0,\dots,k-1\}^N$ and, for every vertex $x_1\dots x_N$, $k$ outgoing edges labeled by $R_0$ to $R_{k-1}$.
The edge labeled by $R_y$ has $x_2\dots x_Ny$ as end vertex.
\end{definition}
Sometimes de Bruijn graphs are defined  as non-oriented graphs.
In the following, we will write $\orB_{N}=\orB_{k,N}$ for the oriented version and $\unB_{N}=\unB_{k,N}$ for the non-oriented one.
Note that in our formalism, the graph $\unB_{N}$ is $\underline{\orB_{N}}$.

De Bruijn graphs are widely seen as representing overlaps between strings of symbols and are also combinatorial models of the Bernoulli map $x\mapsto kx \pmod 1$ and therefore are of interest in the theory of dynamical systems.

It is shown in \cite{MR891925} that each de Bruijn graph $\orB_{N}$ is the line graph (see Definition~\ref{defLine}) of the previous one, $\orB_{N-1}$.
For the sake of completeness we include here a proof of this fact which is crucial for our purposes.
\begin{lemma}[\cite{MR891925}]\label{LemmaLineDeBruijn}
For every $N\in \Nstar$, the de Bruijn graph $\orB_{N+1}$ is (weakly) isomorphic to the line graph of $\orB_{N}$.
\end{lemma}
\begin{proof} 
It is clear from the definition that $\orB_{1}$ is weakly isomorphic to the complete oriented graph on $k$ vertices, with loops.
That is, $\orB_{1}$ has $k$ vertices and for each pair $(v,w)$ of vertices, there is exactly one edge from $v$ to $w$.
In particular, for every $v$ there is a unique edge from $v$ to itself.
It is then obvious that $\orB_{1}$ is weakly isomorphic to the line graph of $\orB_{0}$ (the rose). 

Observe that, for any $N$, for each vertex $v$ in $\orB_{N}$ and each label $R_y$, there is exactly one edge $e$ with initial vertex $v$ and label $R_y$.
Therefore, there is a natural bijection between the vertex set of the line graph of $\orB_{N}$ and the set of couples $\setst{(v,R_y)}{v \tn{ a vertex in }\orB_{N}, 0\leq y<k}$.
Let $v=(x_1\dots x_N)$ be a vertex in $\orB_{N}$. If $N\geq 1$, there is an edge in the line graph from $(v,R_y)$ to $(w,R_z)$ if and only if $w=(x_2\dots x_Ny)$.

We construct now an explicit weak isomorphism $\phi$ from the line graph of $\orB_{N}$ to $\orB_{N+1}$.
We define $\phi$ on the vertices by $\phi(v,R_y)\defegal(x_1\dots x_Ny)$ if $v=(x_1\dots x_N)$.
This is obviously a bijection.
If $N\geq 1$, there is a unique edge in the line graph from $\bigl((x_1\dots x_N),R_y\bigr)$ to $\bigl((x_2\dots x_Ny),R_z\bigr)$ (and all edges are of this form).
Let the image of this edge by $\phi$ be the unique edge in $\orB_{N+1}$ with initial vertex $(x_1\dots x_Ny)$ and label $R_z$ --- see Figure~\ref{HomoBruijn}.
It is straightforward to see that $\phi$ is injective on the set of edges.
Since the two graphs have the same finite number of edges $(k\cdot k^{N+1})$, $\phi$ is also bijective on the set of edges.
Moreover, by definition, $\phi\bigl(\iota(e)\bigr)=\iota\bigl(\phi(e)\bigr)$ for any edge $e$ in the line graph.
Hence, to show that $\phi$ is a weak isomorphism it only remains to check that $\phi\bigl(\tau(e)\bigr)=\tau\bigl(\phi(e)\bigr)$.
If $e$ is an edge from $\bigl((x_1\dots x_N),R_y\bigr)$ to $\bigl((x_2\dots x_Ny),R_z\bigr)$, we have 
\[
	\phi\bigl(\tau(e)\bigr)=(x_2x_3\dots x_Nyz).
\]
On the other hand, $\phi(e)$ has initial vertex $(x_1\dots x_Ny)$ and label $R_{z}$.
Therefore,
\[
	\tau\bigl(\phi(e)\bigr)=(x_2x_3\dots x_Nyz)=\phi\bigl(\tau(e)\bigr).
\]
\end{proof}
\begin{figure}[ht]
\centering
\includegraphics{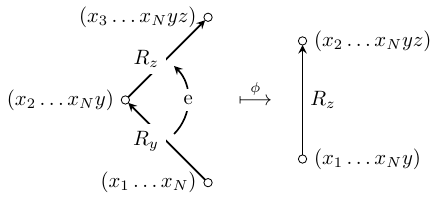}
\caption{The edge $e$ in the line graph of $\orB_{N}$ and its image by $\phi$.}
\label{HomoBruijn}
\end{figure}
%
%
%
%
%
\subsection{Spider-web graphs}
\begin{definition}\label{DefinitionSw}
Let $k\geq 2$. For all $M\in\N$ and  $N\in\Nstar$, the \defi{spider-web graph}  is the labeled oriented graph $\orSw{k}{N}{M}=\orkSw{k}{N}{M}$ with vertex set $\{0,\dots,k-1\}^N\times \{1,2,\dots,M\}$ and for every vertex $v=(x_1\dots x_N,i)$ with $k$ outgoing edges labeled by $R_0$ to $R_{k-1}$.
The edge labeled by $R_y$ has $(x_2\dots x_Ny,i+1)$ as end vertex, where $i+1$ is taken modulo~$M$.
See figure \ref{FigSpiderWeb} for an example.
\end{definition}
\begin{figure}[ht]
\centering
\includegraphics{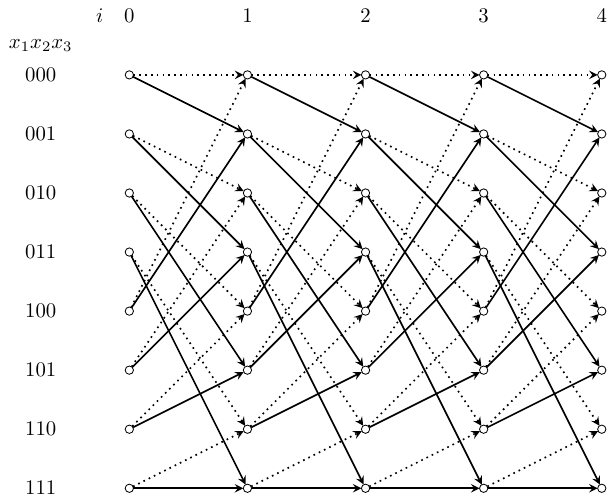}
\caption{
The graph $\orkSw{2}{3}{4}$, where vertices in the slice $i=0$ are identified with vertices in the slice $i=M=4$. Solid edges correspond to edges with label $R_1$ and dotted edges to edges with label $R_0$.}
\label{FigSpiderWeb}
\end{figure}
As with de Bruijn graphs, we write $\orSw{k}{N}{M}$ for the oriented version and $\unSw{k}{N}{M}$ for the non-oriented one.
The vertices of $\unSw{k}{N}{M}$ are partitioned into slices $1,\dots, M$ and edges connect vertices in slice $i$ to vertices in slice $i+1 \pmod M$.
Note that in our definition (unlike papers that talk about spider-web networks), the vertices of the slice $M$ are connected to the vertices of the slice $1$.
Note also that it is possible to similarly define $\orkSw{k}{N}{\infty}$ for all $N$ (with vertex set $\{0,\dots,k-1\}^N\times\Z$).

Observe that the graph $\orkSw{k}{0}{M}$ is a ``thick'' oriented circle (or line if $M=\infty$): the vertex set is $M$ and for every vertex $i$ there are $k$ edges from $i$ to $i+1$.
The graph $\orkSw{1}{0}{M}$ is the usual oriented circle $\vec C_M$.
Therefore,  $\orkSw{k}{0}{1}$ is the rose with one vertex and $k$ oriented edges.

\begin{lemma}\label{lemmaIsoSwTensor}
For all $M\in\barN$ and $N\in\Nstar$ there is a strong isomorphism between $\orSw{k}{N}{M}$ and $\orSw{k}{N}{1}\otimes\vec C_M$.
\end{lemma}
\begin{remark}
Observe that this lemma is not true if we consider the non-oriented spider-web graphs.
This is the main reason why we are brought to work with oriented graphs in this article, even though the final result that we aim at and that we get are about non-oriented graphs (Corollary~\ref{CorConvergenceUnoriented}).
\end{remark}
Lemma~\ref{lemmaIsoSwTensor} together with Theorem~\ref{ThmLimit} ensures that in order to identify the limit of spider-web graphs when $M,N\rightarrow\infty$ it is enough to study the limit of spider-web graphs with $M=1$.
It turns out that the spider-web graphs with $M=1$ are exactly de Bruijn graphs.

Indeed the identification given by $(x_1\dots x_N,1)\mapsto(x_1\dots x_N)$ induces a strong isomorphism between $\orSw{k}{N}{1}$ and $\orB_{N}$.
The isomorphism between non-oriented versions follows.
Hence we have the following.
\begin{proposition}
For all $N\in\Nstar$, the oriented graph $\orSw{k}{N}{1}$ is strongly isomorphic to $\orB_{N}$ and the non-oriented graph $\unSw{k}{N}{1}$ is strongly isomorphic to~$\unB_{N}$.
\end{proposition}
\noindent Lemma~\ref{lemmaIsoSwTensor} directly implies the following.
\begin{corollary}\label{corIsoSwB}
For all $M\in\barN$ and $N\in\Nstar$, $\orSw{k}{N}{M}$ is strongly isomorphic to $\orB_{N}\otimes \vec C_M$.
\end{corollary}
\begin{figure}[ht]
\centering
\includegraphics{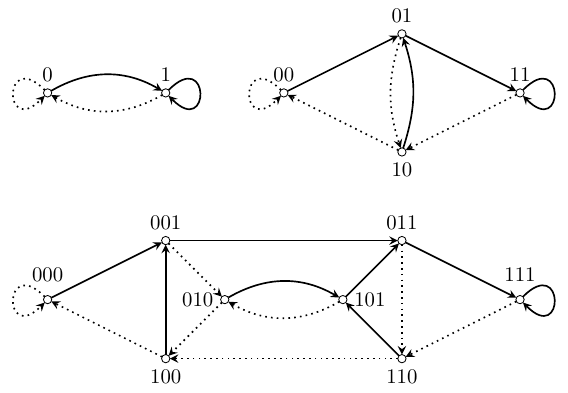}
\caption{Oriented de Bruijn graphs for $k=2$ and $N\in\{1,2,3\}$. Solid edges correspond to edges with label $R_1$ and dotted edges to edges with label $R_0$.}
\label{BruijnGraphs}
\end{figure}
\begin{lemma}\label{LemmaConnected}
For all $M\in\barN$ and $N\in \Nstar$, the graph $\unSw{k}{N}{M}$ is connected.
\end{lemma}
\begin{proof}
By the previous corollary, $\orSw{k}{N}{M}\iso\orB_{N}\otimes \vec C_M$.
On the other side, $\der(\orB_{N})=1$ since there is a loop (labeled by $R_0$) at the vertex $(0\dots 0)$.
Therefore, by Proposition~\ref{PropNbrComponent}, the number of connected components of $\orSw{k}{N}{M}$ is~$1$.
\end{proof}
%
%
%
%
%
\subsection{The group $\lamp_k$ and its action on the $k$-regular rooted tree}\label{SubSecActionLk}
We will use the language of actions on rooted trees, see for instance \cite{MR1841755} and~\cite{MR2162164}.

Fix $k\geq 2$.
The strings over the alphabet $\mathcal A=\{0,1,\dots,k-1\}$ are in one-to-one correspondence with the vertices of the $k$\hyphen regular rooted tree $T_k$ where the root vertex corresponds to the empty string.
Under this correspondence, the $n^\tn{th}$ level of $T_k$ is the set of strings over $\mathcal A$ of length $n$.
The \defi{boundary} $\partial T_k$ of $T_k$ is the set of right infinite strings over $\mathcal A$.
We write $\overline T_k\defegal T_k\cup\partial T_k$.

We also have a one-to-one correspondence between $T_k$ and the ring of polynomials $\Z/k\Z[t]$ given by
\[
	(x_0\dots x_n)\mapsto x_0+x_1t\dots+x_nt^{n}
\]
and a one-to-one correspondence between $\partial T_k$ and the ring of formal series $\Z/k\Z[[t]]$
given by
\begin{equation}\label{eqTwo}
	(x_0x_1x_2\dots)\mapsto \sum_{i\geq 0}x_it^{i}\tag{$\star$}. 
\end{equation}
Let $G\leq \Aut(T_k)$ be a group acting on $T_k$ by automorphisms.
The action is said to be \defi{spherically transitive} if it is transitive on each level.

Extending a result from Grigorchuk and \.Zuk \cite{MR1866850} for $k=2$, Silva and Steinberg showed in \cite{MR2197829} that the lamplighter group $\lamp_k$ introduced in Subsection~\ref{subsectionLamplighter} above acts on $\overline{T}_k$.
They showed that this action is faithful and spherically transitive and described some other interesting properties of this action.
Here, we will look at the action given by
\[
	(x_1x_2x_3\dots)\acts\bar c_r =\bigl((x_1+r)(x_2+x_1)(x_3+x_2)\dots\bigr)
\]
where additions are taken modulo $k$.
This action is slightly different from the one in \cite{MR2197829}, but remains faithful and spherically transitive.

Since the action of $\lamp$ on $T$ is spherically transitive, for any generating set $Y$, for all $N$, the graph of the action on the $N$\textsuperscript{th} level are connected.
Thus, they can be viewed as Schreier graphs $\orSchrei{\lamp}{H_{N}}{Y}$, where $H_{N}=\Stab_{\lamp}(v_N)$, with $v_N$ any vertex of the $N^\tn{th}$ level of $T$.
On the other hand, it is obvious that the graph of the action on $\partial T$ is not connected.
Its connected components correspond to the orbits of the (countable) group $\lamp$ on the (uncountable) set $\partial T$.
They can be viewed as Schreier graphs of the subgroups $\Stab_{\lamp}(\xi)$, $\xi\in\partial T$.
If $v$, $w$ are two vertices of $T$, with $w$ lying on an infinite ray emanating from $v$, then $\Stab_\lamp(w)$ is a subgroup of $\Stab_\lamp(v)$.
This implies that for every $N$, the graph $\orSchrei{\lamp}{H_{N+1}}{Y}$ covers $\orSchrei{\lamp}{H_{N}}{Y}$, and we deduce the following.
\begin{proposition}[\cite{MR3278388},\cite{MR2252901}]\label{PropConv}
For any generating set $Y$ and any $\xi=(x_1x_2\dots)$ ray in $\partial T$,
the sequence of rooted graphs $\bigl(\orSchrei{\lamp}{H_{N}}{Y},(x_1\dots x_n)\bigr)$ converges (as labeled graphs) to 
\[
	\bigl(\orSchrei{\lamp}{\Stab_{\lamp}(\xi)}{Y},\xi\bigr).
\]
\end{proposition}
We also have
\begin{proposition}[\cite{MR3278388},\cite{MR2252901}]\label{PropCayley}
For all but countably many $\xi\in\partial T$, the oriented graph $\orSchrei{\lamp}{\Stab_{\lamp}(\xi)}{Y}$ is strongly isomorphic to $\orCayley{\lamp}{Y}$.
\end{proposition}
\begin{proof}
The oriented labeled graph $\bigl(\orSchrei{\lamp}{\Stab_{\lamp}(\xi)}{Y},\xi\bigr)$ is strongly isomorphic to $\orCayley{\lamp}{Y}$ if and only if $\Stab_{\lamp}(\xi)=\{1\}$.
We will show that $\Stab_{\lamp}(\xi)\neq\{1\}$ for only countably many $\xi$.
In order to prove that, we look at the equivalent action of $\lamp$ on $\Z/k\Z[[t]]$.
Applying formula \eqref{eqTwo} we set for any $F\in\Z/k\Z[[t]]$
\begin{align*}
	F\acts c&=F+1\\
	F\acts b&=(1+t)F.
\end{align*}
Hence for any $i$ we have
\[
	F\acts b^{-i} c b^{i}=F+(1+t)^i.
\]
Let $1\neq g$ be an element in $\lamp$.
Then $g$ admits a unique decomposition as $g=b^ih$ for some $i\in\Z$ and $h\in \A=\langle \{b^jcb^{-j}\}_{j\in Z}\rangle$.
Therefore, there exists $P$ a finite sum of $(1+t)^j$'s, $j\in \Z$, such that for any $F\in\Z/k\Z[[t]]$,
\begin{align*}
	F\acts g&=(1+t)^iF+P&\tn{and}\\
	F\acts g^{-1}&=(1+t)^{-i}F-(1+t)^{-i}P.
\end{align*}
Since the action is faithful (and $g\neq 1$), it is not possible that $i=0$ and $P=0$ together.
Now, suppose $F$ has non-trivial stabilizer.
Then there exists $1\neq g\in \lamp$ such that $F=F\acts g=F\acts g^{-1}$ and therefore $F$ is a solution to the equations
\begin{align*}
	\bigl((1+t)^i-1\bigr)x&=-P\\
	\bigl((1+t)^{-i}-1\bigr)x&=(1+t)^{-i}P.
\end{align*}
We have $i\neq 0$, otherwise we would have $P=0$, which is absurd.
Hence, $F$ is a solution of
\[\label{eqF}
	\bigl((1+t)^i-1\bigr)x=Q,\tag{$\ast$}
\]
with $i>0$ and $Q$ belonging to $L\defegal \Z/k\Z[(1+t),(1+t)^{-1}]$, the subring of $\Z/k\Z[[t]]$ consisting of Laurent polynomials in $1+t$.
Note that $(1+t)^i-1$ is not invertible and thus we cannot write $x=Q/\bigl((1+t)^i-1\bigr)$.

Suppose for a moment that $k$ is prime.
This is equivalent to $\Z/k\Z[[t]]$ being an integral domain.
In this case, given $a\neq 0$ and $b$ in $\Z/k\Z[[t]]$, the equation $ax=b$ has at most one solution.
Now, if $F$ has a non-trivial stabilizer, we had just proved that it satisfies equation \eqref{eqF}.
Since $L$ is countable, there are only countably many equations of this form and, by unicity of solution, countably many solutions of such equation and hence countably many $F$ with non-trivial stabilizer.

If $k$ is not prime, we do not have the unicity of solution of equations in $\Z/k\Z[[t]]$.
For example, for $k=6$ the equation $2t\cdot x=0$ admits uncountably many different solutions (all series $x$ where all coefficients belongs to $\{0,3\}$).
But, in our special case, we claim that \eqref{eqF} has only finitely many solutions.
Using that, we have again that the number of series $F$ with non-trivial stabilizer is countable.

We now prove the claim.
If equation \eqref{eqF} has no solution, then the claim is true.
If there is as at least one solution, the solutions of \eqref{eqF} are in one-to-one correspondence with the solutions of 
\[
	\bigl((1+t)^i-1\bigr)x=0.
\]
But such an equation has only finitely many solutions by the following lemma.
\end{proof}
\begin{lemma}
Let $R$ be a finite ring and $P=\sum_{j=0}^dp_jt^j\in R[t]$ a polynomial.
Then, in $R[[t]]$, the equation $Px=0$ has
\begin{enumerate}
\item only one solution ($x=0$) if the first non-zero coefficient of $P$ does not divide $0$;
\item at most ${\abs R}^d$ solutions if $p_d$ is invertible.
\end{enumerate}
\end{lemma}
\begin{proof}
The first statement is trivially true.

The proof of the second statement is by contradiction.
Observe that if $x=\sum_{n\geq 0} x_nt^n$ is a solution, we have for all $n$
\[
	x_n=-p_d^{-1}\sum_{k=1}^{d} p_{d-k}x_{n+k}.
\]
Now, suppose that the equation $Px=0$ has more than ${\abs R}^d$ solutions.
Choose ${\abs R}^d+1$ different solutions $y_i$.
There exists an integer $l$ such that all the $\bar y_i$ differ, where $\bar x=\sum_{n=0}^lx_nt^n$ is the series $x$ up to degree $l$.
We hence have ${\abs R}^d+1$ distinct polynomials of degree at most $l$, all satisfying the identities above.
But this is not possible.
Indeed, we have at most ${\abs R}^d$ different choices for the coefficients of $t^l$ to $t^{l-d+1}$ and all other coefficients are uniquely determined by these ones.
\end{proof}
Until here, all the results of this subsection were true for any generating set $Y$ of $\lamp_k$.
In the following, we will work with our usual generating set $X_k=\{\bar c_i\}_{i=0}^{k-1}$.
\begin{notation}\label{notationAction}
Denote by $\vec\Gamma_k$ the oriented labeled graph of the action of $\lamp_k$ on $\overline T_k$, with respect to the generating set $X_k$.
If we restrict this action to the $N^\tn{th}$ level of $T_k$, the corresponding oriented labeled graph of the action will be denoted by $\vec\Gamma_{k,N}$.
The graph corresponding to the restriction of the action to $\partial T_k$ will be denoted by $\vec\Gamma_{k,\infty}$.
\end{notation}
\begin{figure}[ht]
\centering
\includegraphics{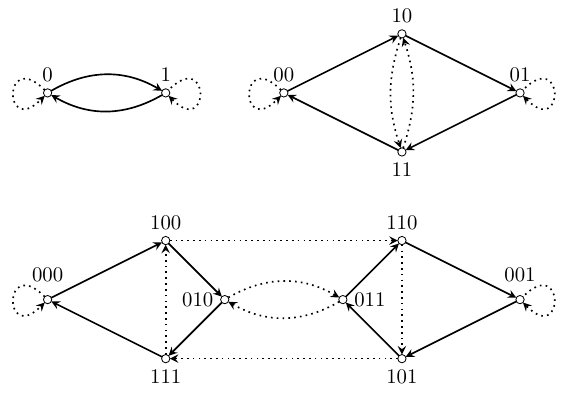}
\caption{The graphs $\vec\Gamma_{2,N}$ for $N\in\{1,2,3\}$. Solid edges correspond to edges with label $\bar c_1$ and dotted edges to edges with label $\bar c_0$.}
\label{SchreierGraphs}
\end{figure}
As with spider-web graphs and de Bruijn graphs, from now on we omit $k$ from our notation and write simply $\lamp$, $T$, $X$, $\vec\Gamma$, $\vec\Gamma_N$ and $\vec\Gamma_\infty$.

The following proposition will, together with Lemma~\ref{LemmaLineDeBruijn}, help to establish a connection between the lamplighter group and the de Bruijn graphs.
\begin{proposition}\label{PropLineAction}
For every $N\in\Nstar$, the graph $\vec\Gamma_{N+1}$ is (weakly) isomorphic to the line graph of $\vec\Gamma_{N}$.
\end{proposition}
\begin{proof}
First, we have that $\vec\Gamma_0$ is the rose with $k$ loops and $\vec\Gamma_1$ is the complete oriented graph (with loops) on $k$ vertices.
Therefore $\vec\Gamma_1$ is weakly isomorphic to the line graph of $\vec\Gamma_0$.

We have that the set of vertices in the line graph of $\vec\Gamma_N$ is in bijection with the set of couples 
\[
\setst{(v,\bar c_i)}{v \tn{ a vertex in }\vec\Gamma_N, 0\leq i<k}.
\]
Let $v=(x_1\dots x_N)$ be a vertex in $\vec\Gamma_N$. If $N\geq 1$, there is an edge in the line graph from $(v,\bar c_i)$ to $(w,\bar c_j)$ if and only if $w=v\acts \bar c_i=\bigl((x_1+i)(x_2+x_1)\dots (x_N+x_{N-1})\bigr)$.

We construct  now an explicit weak isomorphism $\phi$ from the line graph of $\vec\Gamma_N$ to $\vec\Gamma_{N+1}$, see Figure \ref{HomoGamma}.
We define $\phi$ on the vertices by $\phi\bigl((x_1\dots x_N,\bar c_i)\bigr)\defegal(ix_1\dots x_N)$.
It is easy to see that $\phi$ is injective (and hence bijective) on vertices.
If $N\geq 1$, there is a unique edge in the line graph from $\bigl((x_1\dots x_N),\bar c_i\bigr)$ to $\bigl((x_1+i)(x_2+x_1)\dots (x_N+x_{N-1}),\bar c_j\bigr)$ (and all edges are of this form). Let the image of this edge by $\phi$ be the unique edge in $\vec\Gamma_{N+1}$ with initial vertex $(ix_1\dots x_N)$ and label $\bar c_{j-i}$ --- see Figure~\ref{HomoGamma}.
It is straightforward to see that $\phi$ is injective (and thus bijective) on the set of edges.
Moreover, by definition, $\phi\bigl(\iota(e)\bigr)=\iota\bigl(\phi(e)\bigr)$ for any edge $e$ in the line graph.
Hence, to show that $\phi$ is a weak isomorphism it only remains to check that $\phi\bigl(\tau(e)\bigr)=\tau\bigl(\phi(e)\bigr)$.
If $e$ is an edge from $\bigl((x_1\dots x_N),\bar c_i\bigr)$ to $\bigl((x_1+i)(x_2+x_1)\dots (x_N+x_{N-1}),\bar c_j\bigr)$, we have 
\[
	\phi\bigl(\tau(e)\bigr)=\bigl(j(x_1+i)(x_2+x_1)\dots (x_N+x_{N-1})\bigr).
\]
On the other hand, $\phi(e)$ has initial vertex $(ix_1\dots x_N)$ and label $\bar c_{j-i}$.
Therefore,
\[
	\tau\bigl(\phi(e)\bigr)=\bigl((i+j-i)(x_1+i)(x_2+x_1)\dots (x_N+x_{N-1})\bigr)=\phi\bigl(\tau(e)\bigr).
\]
\end{proof}
\begin{figure}[ht]
\centering
\includegraphics{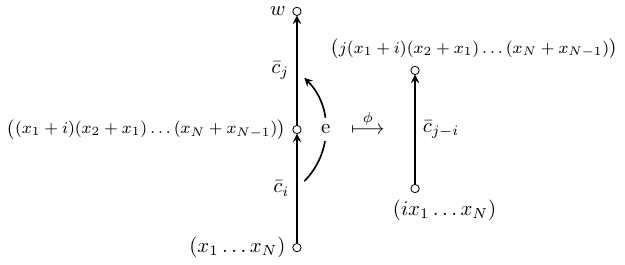}
\caption{The edge $e$ in the line graph of $\vec\Gamma_N$ and its image by $\phi$.}
\label{HomoGamma}
\end{figure}
%
%
%
%
%
\subsection{Convergence of the spider-web graphs to the Cayley graph of the lamplighter group}\label{sectionSpiderLamplighter}
In this subsection, we use results of the last two subsections to finally establish a link between the spider-web graphs $\orkSw{k}{N}{M}$ and the graph $\vec\Gamma_k$ of the action of the lamplighter group $\lamp_k$ on the $k$-regular tree $\overline T_k$ and to prove our main results.
\begin{theorem}\label{ThmWeakIsomorphism}
Let $k\geq 2$.
For all $N\in\Nstar$, the de Bruijn graph $\orB_{k,N}$ is weakly isomorphic to $\vec\Gamma_{k,N}$, the graph of the action of $\lamp_k$ on the $N^\tn{th}$ level of $T_k$, with respect to the generating set $X_k=\{\bar c_r\}_{r=0}^{k-1}$ (see \eqref{presentation2} on page \pageref{presentation2}).
\end{theorem}
\begin{proof}
The proof goes by induction on $N$.
For $N=0$, both graphs are weakly isomorphic to the rose with $k$ loops.
For $N\geq 1$ we have by Lemma~\ref{LemmaLineDeBruijn} that $\orB_{k,N}$ is weakly isomorphic to the line graph of $\orB_{k,N-1}$.
Since $\orB_{k,N-1}$ is (by induction) weakly isomorphic to $\vec\Gamma_{k,N-1}$ and that the line graph does not depend on the labeling we have a weak isomorphism between $\orB_{k,N}$ and the line graph of $(\vec\Gamma_k)_{N-1}$.
By Proposition~\ref{PropLineAction}, this graph is itself weakly isomorphic to $(\vec\Gamma_k)_N$.
\end{proof}
\begin{remark}
These graphs are not strongly isomorphic.
Indeed, in $\orB_{k,N}$ for every vertex $v$, all edges ending at $v$ have the same label, but in a Schreier graph two edges having the same end-vertex have distinct labels. (See Figures~\ref{BruijnGraphs} and~\ref{SchreierGraphs}.)
\end{remark}
Theorem~\ref{ThmWeakIsomorphism} and the fact that $\vec\Gamma_{k,N}$ covers $\vec\Gamma_{k,N-1}$ (see the discussion before Proposition~\ref{PropConv}),  imply the following property of de Bruijn graphs that, as far we know, was not observed before.
\begin{corollary}\label{corCovering}
For all $N\in\N$ the graph $\orB_N$ (weakly) covers $\orB_{N-1}$.
\end{corollary}
We are now able to prove our main theorem.
\begin{theorem}\label{ThmConvergenceM1}
Let $k\geq 2$.
Recall that $\lamp_k$ denotes the lamplighter group $(\Z/k\Z)\wr\Z$ and $X_k=\{{\bar c_r}\}_{r=0}^{k-1}$ is the generating system given by \eqref{presentation2}, see page~\pageref{presentation2}.
\begin{enumerate}
\item
The unlabeled oriented de Bruijn graphs $\orB_{k,N}$ converge to $\orCayley{\lamp_k}{X_k}$ in the sense of Benjamini-Schramm convergence (see Definition~\ref{DefBSConv}).
\item
The following diagram commutes, where the arrows stand for Benjamini-Schramm convergence of unlabeled graphs.

\centering
\includegraphics{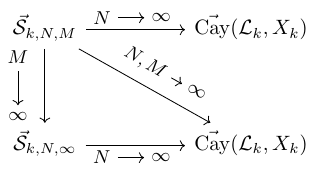}
\end{enumerate}
\end{theorem}
\begin{proof}
By Proposition~\ref{PropCayley}, for all but countably many $\xi=(x_1x_2\dots)$ in $\partial T_k$, the oriented graph $(\vec\Gamma_{k,\infty},\xi)^0$ is strongly isomorphic to the oriented Cayley graph  $\orCayley{\lamp_k}{X_k}$.
On the other hand, by Proposition~\ref{PropConv}, the graphs $\bigl(\vec\Gamma_{k,N},(x_1\dots x_N)\bigr)$ strongly converge to $(\vec\Gamma_{k,\infty},\xi)^0$.
Theorem~\ref{ThmWeakIsomorphism} gives us a weak isomorphism between $\orB_{k,N}$ and $\vec\Gamma_{k,N}$.
Therefore, $\orB_{k,N}$ converge to $\orCayley{\lamp_k}{X_k}$.
This ends the proof of the first part of the theorem.

In order to prove the second part of the theorem, we should consider an auxiliary diagram, see Figure~\ref{LimiteSpiderPreuve}.
First note that we already know that when $N\to\infty$ de Bruijn graphs $(\orB_{k,N},v_{N})$ weakly converge to the Cayley graph $(\orCayley{\lamp_k}{X_k},1_{\lamp_k})$ for nearly all choices of the $v_{N}$ and that it is obvious that $(\vec C_M,0)$ weakly converge to $(\vec\Z,0)$ when $M\to \infty$.
Hence, by Theorem~\ref{ThmLimit}, the
 diagram in Figure
\ref{LimiteSpiderPreuve}
is commutative.
Finally, since this statement is true for nearly all choices of roots, we have the convergence in the sense of Benjamini-Schramm when we choose the roots uniformly.
\begin{figure}[ht]
\centering
\includegraphics{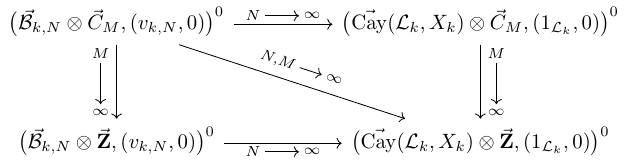}
\caption{Limit of de Bruijn graphs}
\label{LimiteSpiderPreuve}
\end{figure}
By Proposition~\ref{PropPresentation2}, for every $M\in\barN$, any connected component of $\orCayley{\lamp_k}{X_k}\otimes\vec C_M$ is strongly isomorphic to $\orCayley{\lamp_k}{X_k}$; in particular, this does not depend on $M$.
On the other hand, for all $M\in\barN$, $\orB_{k,N}\otimes\vec C_M$ is strongly isomorphic to $\orkSw{k}{N}{M}$ by Corollary~\ref{corIsoSwB}, and therefore is connected.
\end{proof}
Using Proposition~\ref{ConvergenceOrientedVsNonOriented}, we obtain the same result for non-oriented versions of de Bruijn and spider-web graphs as an immediate corollary.
\begin{corollary}~\label{CorConvergenceUnoriented}
\begin{enumerate}
\item
The unlabeled de Bruijn graphs $\unB_{k,N}$ converge to $\unCayley{\lamp_k}{X_k}$ in the sense of Benjamini-Schramm convergence (see Definition~\ref{DefBSConv}).
\item
The following diagram commutes, where the arrows stand for Benjamini-Schramm convergence of unlabeled graphs.

\centering
\includegraphics{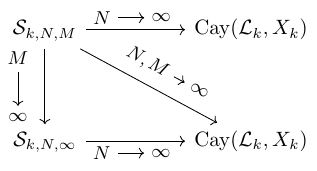}
\end{enumerate}
\end{corollary}
\begin{remark}
Remark~\ref{RmkLamplighter} implies that in Theorem~\ref{ThmConvergenceM1} and its Corollary~\ref{CorConvergenceUnoriented}, $\lamp_k$ can be replaced by any wreath product $G\wr \Z$, where $G$ is a finite group of cardinality $k$.
However, Theorem~\ref{ThmWeakIsomorphism} does not hold necessarily if $G$ is not abelian.
\end{remark}
%
%
%
%
%
%
%
\section{Computation of spectra}\label{sectionSpectra}
In this section we compute the characteristic polynomial and  the spectrum of the adjacency matrix of $\unSw{k}{N}{M}$ for all $M,N\in \N$.
The spectra of the graphs $\{\Gamma_N\}$ of the action of $\lamp_2$ on the levels of the binary rooted tree were first computed by Grigorchuk and \.Zuk in \cite{MR1866850} using the fact that they form a tower of coverings, see the discussion just before Proposition~\ref{PropConv}.
(Note that the multiplicity in their formula is not completely correct -- compare with our Theorem~\ref{thmMultiplicity} below.)
These computations were extended to any wreath product $G\wr \Z$, with $G\neq\{1\}$ finite, by Kambites, Silva and Steinberg in \cite{MR2252901}, using automata theory.
Dicks and Schick \cite{MR1934693} computed the spectral measures for random walks on $G\wr \Z$ using entirely different methods (see also \cite{MR2131635}).

On the other hand, Delorme and Tillich computed the spectra of $\orB_N$ in \cite{MR1621013} using simple matrix transformations.
It is well known that for any pair of square matrices $A$ and $B$ with respective eigenvalues $(\lambda_i)_{i=1}^n$ and $(\mu_j)_{j=1}^m$, the spectrum of $A\otimes B$ is $\setst{\lambda_i\mu_j}{1\leq i\leq n,1\leq j\leq m}$.
However, this formula cannot be applied in our case since we want to compute the spectrum of $\underline{\orB_N\otimes\vec C_M}$ and there is no formula relating eigenvalues of a matrix $A$ (the adjacency matrix of an oriented graph) and of its symmetrized matrix $A+A^*$ (the adjacency matrix of the underlying graph).
Instead we generalize Delorme and Tillich method directly to all $\unSw{k}{N}{M}$.

For $A$ a matrix, we denote its characteristic polynomial by $\charac(A)$.
For a graph (oriented or not, etc.) $\Gamma$, the \defi{characteristic polynomial} $\charac(\Gamma)$ of $\Gamma$ is the characteristic polynomial of its adjacency matrix.
The following lemma summarizes discussions 2.(1), 2.(2) and 2.(3) from~\cite{MR1621013}.
\begin{lemma2}\label{lemmaCharacPoly}
Let $\vec\Gamma$ be an oriented graph, $\underline{\vec\Gamma}$ the underlying non-oriented graph and $\vec A$ and $A$ their respective adjacency matrices.
Suppose that there exist complex matrices $D$ and $U$ with $U$ unitary such that $\vec A=U^*DU$.
Then $\charac(\vec \Gamma)=\charac(D)$.
Moreover, $A=U^*(D+D^*)U$ and $\charac(\underline{\vec\Gamma})=\charac(D+D^*)$.
\end{lemma2}
\begin{proof}
We have $A=\vec A+(\vec A)^*=U^*DU+U^*D^*U=U^*(D+D^*)U$.
On the other hand, $A$ and $D$ are equivalent matrices and therefore have the same characteristic polynomial.
\end{proof}
Observe that for a given complex matrix $D$ we can construct an oriented weighted graph $\vec\Theta$ with adjacency matrix $D$, where a \defi{weighted graph} is a graph $\Gamma=(V,E)$ (oriented or not, labeled or not, etc.) with a map $w\colon E\to\field C$ which assign to each edge a complex number (a weight) such that $w(\bar e)$ is the complex conjugate of $w(e)$.
In this case, $D+D^*$ is the adjacency matrix of $\underline{\vec\Theta}$.

In their article, Delorme and Tillich use this to compute the spectrum of de Bruijn graphs $\unB_{N}$ and prove the following.
\begin{proposition2}[Dellorme-Tillich]\label{PropSpecBruijn}
For all $N\in\Nstar$, let $\vec\Theta_{k,N}$ be the weighted oriented graph which is the disjoint union of
\begin{enumerate}
\item one oriented loop,
\item for all $0\leq i\leq N-2$, $(k-1)^2k^{N-i-2}$ disjoint oriented paths of length $i$,
\item $k-1$ disjoint oriented paths of length $N-1$;
\end{enumerate}
where all edges have weight $k$ and an oriented path of length $i$ is the oriented graph with vertex set $V\defegal\{0,\dots ,i\}$ and for every vertex $j\in V$ a unique edge from $j$ to $j+1$ (see Figure~\ref{orientedPath}).
Let $\vec D_{N}$ be the adjacency matrix of this graph and $\vec B_{N}$ be the adjacency matrix of $\orB_{N}$.
Then there exists $U=U_{N}$ unitary with $\vec B_{N}=U^*\vec D_{N}U$.
\end{proposition2}
\begin{figure}[ht]
\centering
\includegraphics{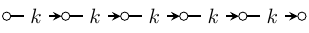}
\caption{An oriented weighted path of length $5$, where all edges have weight $k$.}
\label{orientedPath}
\end{figure}
\begin{figure}[ht]
\centering
\includegraphics{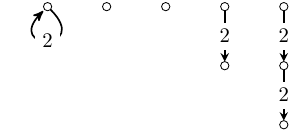}
\caption{The oriented weighted graph $\vec\Theta_{2,3}$.}
\label{Theta23}
\end{figure}
See Figure \ref{Theta23} for the example of $\vec\Theta_{2,3}$.

The only thing remaining to do in order to compute the characteristic polynomial of $\orSw{k}{N}{M}$ is to express the adjacency matrix $\vec S_{N,M}$ of $\orSw{k}{N}{M}$ using $\vec B_{N}$.
But it is well known that, for non-oriented graphs, the adjacency matrix of a tensor product is the tensor (or Kronecker) product of adjacency matrices.
This is also trivially true for oriented graphs.
Therefore, we have
\[
	\vec S_{N,M}=\vec B_{N}\otimes \vec A_M,
\]
where $\vec A_M$, the adjacency matrix of the oriented cycle $\vec C_M$, has a $1$ in position $(i,j)$ if and only if $j\equiv i+1\pmod M$.
Denoting $\vec D_{N,M}\defegal \vec D_{N}\otimes \vec A_M$, a simple computation gives us
\begin{align*}
	\vec S_{N,M}&=\vec B_{N}\otimes \vec A_M\\
	&=(U^*\vec D_{N}U)\otimes(\Id^*\vec A_M\Id)\\
	&=(U^*\otimes \Id^*)(\vec D_{N}\otimes \vec A_M)(U\otimes\Id)\\
	&=(U\otimes \Id)^*(\vec D_{N,M})(U\otimes\Id),
\end{align*}
where $\Id$ is the identity matrix of size $M$.
Since $U$ and the identity matrix are both unitary, their tensor product $U\otimes\Id$ is also unitary.
Thus, by Lemma~\ref{lemmaCharacPoly}, the characteristic polynomial of $\unSw{k}{N}{M}$ is equal to $\charac(\vec D_{N,M}+\vec D_{N,M}^*)$.

On the other hand, $\vec D_{N,M}$ is the adjacency matrix of the weighted graph $\vec\Theta_{k,N,M}\defegal\vec\Theta_{k,N}\otimes \vec C_M$.
Computing this tensor product we have that the weighted graph $\vec\Theta_{k,N,M}$ is the disjoint union of
\begin{enumerate}
\item
	one oriented cycle of length $M$,
\item
	for all $0\leq i\leq N-2$, $M(k-1)^2k^{N-i-2}$ disjoint oriented paths of length~$i$,
\item
	$M(k-1)$ disjoint oriented paths of length $N-1$;
\end{enumerate}
where all edges have weight $k$ --- see Figure~\ref{ExTheta} for an example.
Hence, $\vec D_{N,M}+\vec D_{N,M}^*$ is the adjacency matrix of $\underline{\vec\Theta_{k,N,M}}$.
Therefore
\begin{align*}
	\charac(\unSw{k}{N}{M})&=\charac(\vec D_{N,M}+\vec D_{N,M}^*)\\
	&=Q(x)\cdot P_{N}(x)^{M(k-1)}\prod_{i=1}^{N-1} P_{i}(x)^{M(k-1)^2k^{N-i-1}},
\end{align*}
where $Q(x)=Q_{k,M}(x)$ is the characteristic polynomial of the non-oriented cycle of length $M$ with all edges of weight $k$ and $P_{i}(x)=P_{i,k}(x)$ is the characteristic polynomial of the non-oriented path of length $i-1$ with all edges of weight $k$.
\begin{figure}[ht]
\centering
\includegraphics{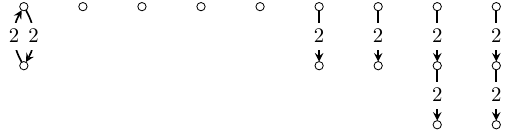}

\caption{The oriented weighted graph $\vec\Theta_{2,3,2}$.}
\label{ExTheta}
\end{figure}
We now want to have an explicit form for the $P_i$'s and $Q$.
In \cite{MR1621013}, Delorme and Tillich showed that $P_i(x)=k^iV_i(x/k)$ with $V_i$ the Chebyshev polynomial of the second kind of degree $i$.
The set of roots of $P_i(x)$ is exactly
\[
	\setst{2k\cos\Bigl(\frac{t\pi}{i+1}\Bigr)}{1\leq t\leq i}
\]
and all roots are simple.
On the other hand, $Q(x)$ is the characteristic polynomial of the adjacency matrix of a non-oriented cycle of length $M$ with edges of weight $k$.
If $M\neq 1$, all the non-zero entries of this matrix have value $k$ and are in position $(i,j)$ with $\abs{i-j}=1$.
If $M=1$, the cycle is a loop of length $k$ and the adjacency matrix consists of a unique entry: $2k$.
In both cases, the adjacency matrix is a circulant matrix of size $M$ and it has characteristic polynomial
\[
	Q(x)=\prod_{l=1}^M (x-2k\cos\Bigl(\frac{2\pi l}{M}\Bigr)).
\]
That is, the root $2k$ has multiplicity $1$, the root $-2k$ has multiplicity $1$ if $M$ is even and multiplicity $0$ otherwise and for all $1\leq l<M/2$ the root $2k\cos(\frac{2\pi l}{M})$ has multiplicity $2$.
Therefore we have just proved the following.

For every $k\geq 2$, $N\in\Nstar$, and $M\in\N$, the spectrum of $\unSw{k}{N}{M}$ consist of $2k$ with multiplicity $1$,  of
\[
	\setst{2k\cos\Bigl(\frac {p}q\pi\Bigr)}{1\leq p<q\leq N+1; p\tn{ and }q\tn{ relatively prime}},
\]
with multiplicity not specified yet and, if $M$ is even, also of $-2k$ with multiplicity~$1$.

The computation of the multiplicity of $2k\cos\Bigl(\frac {p}q\pi\Bigr)$ for a given $p$ and $q$ is done in four steps.
Step one: compute its multiplicity in eigenvalues (interpreted as roots) of $Q(x)$; it is either $0$ or $2$.
Step two: compute its multiplicity in $P_N(x)^{M(k-1)}$; it is either $0$ or $M(k-1)$.
Step three: for all $1\leq i\leq N-1$, compute its multiplicity in $P_i(x)^{M(k-1)^2k^{N-i-1}}$; it is either $0$ or $M(k-1)^2k^{N-i-1}$.
Step four: add the results of the three previous steps.

In step one, the multiplicity is non-zero if and only if there exists $1\leq l<\frac M2$ such that $\cos(2\pi l/M)=\cos(p\pi/q)$.
But this is possible if and only if $l=Mp/2q\geq 1$.
Since $l$ is an integer and $p$ and $q$ are relatively prime, the multiplicity is $2$ if and only if $2q$ divides $Mp$.

In step two, the multiplicity is non-zero if and only if $t=p(N+1)/q$, if and only if $q$ divides $N+1$.

In step three, the multiplicity is non-zero if and only if $t=p(i+1)/q$, if and only if $q$ divides $i+1$.

Summing up all these quantities we conclude that for $1\leq p<q\leq N+1$, the multiplicity of $2k\cos\Bigl(\frac {p}q\pi\Bigr)$ in the spectrum of $\unSw{k}{N}{M}$ is
\[
M(k-1)^2\cdot\Biggl(\sum_{j=1}^{\lfloor\frac Nq\rfloor}k^{N-jq}\Biggr)+M(k-1)r_1+2r_2,
\]
with 
\[
r_1=r_1(q,N)=\begin{cases}1& \tn{if $q$ divides $N+1$}\\0& \tn{otherwise}\end{cases}\]
and
\[ r_2=r_2(p,q,M)=\begin{cases}1& \tn{$2q$ divides $Mp$}\\0& \tn{otherwise.}\end{cases}
\]

Observe that in the above sum, the first summand is equal to
\begin{align*}
	M(k-1)^2\cdot\Biggl(\sum_{j=1}^{\lfloor\frac Nq\rfloor}k^{N-jq}\Biggr)&=M(k-1)^2k^N\cdot\sum_{j=1}^{\lfloor\frac Nq\rfloor}k^{-jq}\\
	&=M(k-1)^2k^N\left(\frac{1-k^{-q(\lfloor\frac Nq\rfloor+1)}}{1-k^{-q}}-1\right).
\end{align*}
In the case where $q=N+1$, the multiplicity of $2k\cos(\frac p{N+1}\pi)$ is $M(k-1)+2r_2$.
If $M=1$, $q$ cannot divide $M$ ($q$ is at least $2$), thus in this case $r_2$ is always equal to $0$.

If $\Gamma$ is a finite graph with $m$ vertices and with eigenvalues of the adjacency matrix $\lambda_1\geq\dots\geq\lambda_m$, we write
\[
	\mu_\Gamma=\frac 1m\sum_{i=1}^{m}\delta_{\lambda_i}
\]
for the \defi{spectral measure} on $\Gamma$,
where $\delta_x$ denotes the Dirac mass on $x$.
Then we have the following.
\begin{theorem2}\label{thmMultiplicity}
The spectral measure $\mu_{\unSw{k}{N}{M}}$ of $\unSw{k}{N}{M}$ is, if $M$ is odd, 
\begin{multline*}
	\frac{1}{Mk^N}\delta_{2k}\\
		+\sum \delta_{2k\cos(\frac pq\pi)}\left((k-1)^2\bigg(\frac{1-k^{-q(\lfloor\frac Nq\rfloor+1)}}{1-k^{-q}}-1\bigg)+\frac{k-1}{k^N}r_1+\frac{2}{Mk^N}r_2\right)
\end{multline*}
where the sum is over all $1\leq p<q\leq N+1$ with $(p,q)=1$.

If $M$ is even, there is one more summand: $\frac{1}{Mk^N}\delta_{-2k}$.
\end{theorem2}
\begin{remark2}\label{RmqIndepM}
It directly follows from the formula in the above theorem that for $k$ and $N$ fixed, $\unSw{k}{N}{M}$ and $\unSw{k}{N}{M'}$ have the same spectrum, except maybe for the value $-2k$, and that 
the total variation distance between $\mu_{\unSw{k}{N}{M}}$ and $\mu_{\unSw{k}{N}{M'}}$ is bounded by $\frac{2}{k^N}$, independently from $M$ and $M'$.
\end{remark2}
Since spider-web graphs converge, in the sense of Benjamini-Schramm, to the Cayley graph $\unCayley{\lamp_k}{X_k}$ we retrieve the Kesten spectral measure of the graph $\unCayley{\lamp_k}{X_k}$.
This measure was first computed by Grigorchuk and \.Zuk in \cite{MR1866850} for $k=2$ and then by Dicks and Schick in \cite{MR1934693} and by Kambites, Silva and Steinberg in \cite{MR2252901} for the more general case $G\wr \Z$, with $G\neq\{1\}$ a finite group.
\[
	\mu_\unCayley{\lamp_k}{X_k}=(k-1)^2\sum_{q\geq2}\frac{1}{k^q-1}\bigg(\sum_{\substack{1\leq p<q\\ (p,q)=1}}\delta_{2k\cos(\frac pq\pi)}\bigg).
	\tag{\#}
	\label{SpectralMeasureInfinite}
\]
%
%
%
%
%
\section{General results on spider-web graphs}\label{sectionGeneralResults}
In~\ref{ThmWeakIsomorphism} we proved that $\orSw{k}{N}{1}$ are weakly isomorphic to Schreier graphs of the lamplighter group $\lamp$.
Other spider-web graphs are so far described as $\orSw{k}{N}{M}\iso\orSw{k}{N}{1}\otimes\vec C_M$.
In this section we will show that $\orSw{k}{N}{M}$ is also a Schreier graph of $\lamp$ for each $M,N$.
Then we characterize which of the $\orSw{k}{N}{M}$ are transitive.
Finally, we generalize to spider-web graphs some statements that are known for de Bruijn graphs: existence of Eulerian and Hamiltonian paths, the property of being a line graph and some facts about covering.

As before we fix a $k\geq 2$ and omit to write it when it is not necessary.
We will write $\Nstar$ for $\{0,1,2,\dots\}$ and $\barN$ for $\{1,2,\dots,\infty\}$.
We also take $X=\{\bar c_i\}_{i=0}^{k-1}$ (see \eqref{presentation2} on page \pageref{presentation2}) as generating set for the lamplighter group $\lamp$.
%
%
%
%
%
\subsection{Spider-web graphs as Schreier graphs of lamplighter groups}
In Theorem~\ref{ThmWeakIsomorphism} we proved that the de Bruijn graph $\orB_{N}$ is weakly isomorphic to a Schreier graph of $\lamp$ by using line graphs.
Let us denote $W_{N,1}\defegal\Stab_\lamp(0^N)$.
Then we have the following.
\begin{theorem}\label{ThmWeakIsomorphismSpider}
For all $M\in\barN$ and $N\in\Nstar$, the spider-web graph $\orSw{k}{N}{M}$ is weakly isomorphic to
\[
	\orSchrei{\lamp}{W_{N,M}}{X}
\]
where $W_{N,M}=\setst{g\in W_{N,1}}{\exp_X(g)\equiv 0\pmod M}$.
\end{theorem}
\begin{proof}
Since $\orB_{N}$ is weakly isomorphic to $\vec\Gamma_N$, the spider-web graph $\orSw{k}{N}{M}\iso\orB_{N}\otimes\vec C_M$ is weakly isomorphic to $\vec\Gamma_N\otimes\vec C_M$ which is a Schreier graph of $\lamp$ by Proposition~\ref{propSchreierTensor} and the description of $W_{N,M}$ follows.
\end{proof}
\begin{remark}
Geometrically, we are here in the situation described in Remark~\ref{rmkGeo}, where the action of $\lamp$ on $\vec C_M$ is given by $j\acts (\prod_{t=1}^j (b^{i_t}cb^{-i_t})^{r_t}\cdot b^{r_b})=j+r_b \pmod M$.
Indeed, $b^{i_t}cb^{-i_t}=\bar c_0^{i_t-1}\bar c_1\bar c_0^{-i_t}$ and therefore we have  $\exp_X(\prod_{t=1}^j (b^{i_t}cb^{-i_t})^{r_t}\cdot b^{r_b})=r_b$.
\end{remark}
Given a graph, a group and a generating set, there could be a priori many different ways to represent the graph as a Schreier graph of the group.
It is easy to check that every $g\neq 1$ in $\lamp$, can be written in a unique way as $\prod_{t=1}^j (b^{i_t}cb^{-i_t})^{r_t}\cdot b^{r_b}$, where $j,r_b,r_t,i_t\in \Z$ with $j\geq 0$ and $0<r_t<k$ for all $1\leq t\leq j$.
This allows us to define a subgroup $H_{N,M}$  of $\lamp$ as the following set:
\[
	H_{N,M}\defegal\left\{g=\prod_{t=1}^j (b^{i_t}cb^{-i_t})^{r_t}\cdot b^{r_b}\in\lamp\,\middle|\,
		\begin{gathered}
		r_b \equiv0 \pmod M\\
		\forall i, 1\leq i\leq N, \sum\limits_{i_t\stackrel{N}{\equiv} i}r_t\equiv 0 \pmod k
		\end{gathered}
	\right\},
\]
where the second sum is over all $i_t\equiv i \pmod N$.
\begin{theorem}\label{ThmSwSchreier}
For all $M\in\barN$ and $N\in\Nstar$, the spider-web graph $\unSw{k}{N}{M}$ is weakly isomorphic to
$\unSchrei{\lamp}{H_{N,M}}{X^{-1}}$.
\end{theorem}
\begin{proof}
Define the following permutations on the vertex set $V$ of $\orSw{k}{N}{M}$:
\begin{align*}
	(x_1\dots x_N,j)\acts b&\defegal (x_Nx_1\dots x_{N-1},j-1)\\
	(x_1\dots x_N,j)\acts c&\defegal\bigl((x_1-1)x_2\dots x_N,j\bigr)
\end{align*}
where $x_1-1$ is taken modulo $k$ and $j-1$ modulo $M$.
The group $G$ generated by $b$ and $c$ acts on $V$.

An easy check shows that $b$ and $c$ satisfies the relations in the presentation \eqref{presentation1} (page \pageref{presentation1}) of $\lamp$. 
Therefore $G$ is a quotient of $\lamp$, which implies that $\lamp$ acts too on $V$.
Moreover, for the generating set $X^{-1}=\{\bar c_i^{-1}\}$ there are exactly $k$ edges in the graph of the action with initial vertex $(x_1\dots x_N,j)$: the one labeled by $(\bar c_r)^{-1}=c^{-r}b^{-1}$ having $(x_2\dots x_N(x_1+r),j+1)$ as end vertex.
On other hand, in $\orSw{k}{N}{M}$ there are also exactly $k$ edges with initial vertex $(x_1\dots x_N,j)$: the one labeled by $R_y$ having $(x_2\dots x_Ny,j+1)$ as end vertex.
Since vertex sets and adjacency relations (if we forget about labeling) are the same, $\orSw{k}{N}{M}$ is weakly isomorphic to the graph of the action of $\lamp$ on $V$ with respect to the generating set $X^{-1}$.
Moreover, this graph being connected, it is also weakly isomorphic to the Schreier graph $\orSchrei{\lamp}{H_{N,M}}{X^{-1}}$, with $H_{N,M}=\Stab_{\lamp}(0\dots0,0)$.
A straightforward calculation gives us $H_{N,M}$.

Finally, since we have an isomorphism between oriented graphs, there is an isomorphism between the underlying non-oriented graphs.
\end{proof}

In \cite{MR3278388} Grigorchuk and Kravchenko classified subgroups of $\lamp$ and gave a criterion for normality.

We now recall these two results and identify $H_{N,M}$ and $W_{N,M}$ according to this classification.
We then are able to see which of the subgroups are normal (in which case the corresponding Schreier graphs are in fact Cayley graphs).

Recall that $\A=\oplus_{\Z}\Z/k\Z$ is the abelian part of $\lamp$ and that $b$ acts on $\A$ by shift.
\begin{lemma}[\cite{MR3278388}, Lemma 3.1]
Let $H$ be a subgroup of $\lamp$. Then it defines the triple $(s,H^0,v)$ where $s\in\N$ is such that $s\Z$ is the image of the projection of $H$ on $\Z$, $H^0=H\cap \A$, satisfying $H^0\acts b^s=H^0$, and $v\in \A$ is such that $vb^s\in H$.
The $v$ is uniquely defined up to addition of elements from $H^0$. For $s=0$ one can choose $v=1_\lamp$.

Conversely any triple $(s,H^0,v)$ with such properties gives rise to a subgroup of $\lamp_k$. Two triples $(s,H^0,v)$ and $(s',G^0,v')$ define the same subgroup if and only if $s=s'$, $H^0 =G^0$ and $vH^0 =v'G^0$.
Moreover, $H \subseteq G$ if and only if $s'|s$, $H^0 \subseteq G^0$ and $v \equiv \prod_{i=0}^{s/s'}({v'}\acts b^{is'}) \pmod{G^0}$.
\end{lemma}
\begin{lemma}[\cite{MR3278388}, Lemma 3.2]
Let $H$ be a subgroup of $\lamp$. Then $H$ is normal if and only if the corresponding triple $(s, H^0, v)$, satisfies the additional properties that $H^0\acts b = H^0$, $v(v\acts b)^{-1} \in H^0$ and $g(g\acts b^s)^{-1}\in H^0$ for all $g\in \A$.
\end{lemma}
Note that in \cite{MR3278388} only the case of $k$ prime is treated.
However both lemmas remain true for all $k$.
\begin{proposition}~\label{PropHW}
\begin{enumerate}
\item
For all $M\in\N$ and $N\in\Nstar$ the subgroup $H_{N,M}$ corresponds to the triple $(M,H^0,1_\lamp)$, where $H^0={H_{N,M}}^0$ is the following subgroup of $\A$
\[
H^0=\left\{\prod_{t=1}^j (b^{i_t}cb^{-i_t})^{r_t}\in\lamp\,\middle|\,
\forall i, 1\leq i\leq N, \sum\limits_{i_t\stackrel{N}{\equiv} i}r_t\equiv 0 \pmod k
\right\}.
\]
For all $N\in\Nstar$ the subgroup $H_{N,\infty,}$ corresponds to the triple $(0,H^0,1_\lamp)$.
\item
For all $M\in\N$ and $N\in\Nstar$ the subgroup $W_{N,M}$ corresponds to the triple $(M,W^0,1_\lamp)$, with $W^0={W_{N,M}}^0=W_{N,1}\cap \A$. In particular, $W^0$ does not depend on $M$.
For all $N\in\Nstar$ the subgroup $W_{N,\infty}$ corresponds to the triple $(0,W^0,1_\lamp)$.
\end{enumerate}
\end{proposition}
\begin{proof}
We first prove the proposition for $H_{N,M}$.
It is obvious that $H\cap \A=H^0$.

Take $g=\prod_{t=1}^j (b^{i_t}cb^{-i_t})^{r_t}\cdot b^{r_b}$ in $H_{N,M}$.
The projection of $g$ onto $\Z$ is $r_b$, which is equal to $0 \pmod M$.
If $M$ is finite, since $b^M$ belongs to $H_{N,M}$, the projection of $H_{N,M}$ onto $\Z$ is $M\Z$.
If $M=\infty$, then the projection of $H_{N,M}$ onto $\Z$ is $0$.
Finally, $1_\lamp b^M$ belongs to $H_{N,M}$ as asked.

Now, for $W_{N,M}$ we first look at the case $M=1$.
Since $b$ stabilizes $(0\dots 0)$, it belongs to $W_{N,M}$, the stabilizer of $(0\dots 0)$.
Therefore, $W_{N,1}$ corresponds to the triple $(1,W^0,1_\lamp)$.
For an arbitrary $M$, $W_{N,M}$ is the subgroup of $W_{N,1}$ consisting of elements with total exponent equal to $0$ modulo $M$.
Since  $b^icb^{-i}=b^{i-1}\bar c_1b^{-i}$, we have $\exp_{X_k}(b^icb^{-i})=0$ and for any $g=\prod_{t=1}^j (b^{i_t}cb^{-i_t})^{r_t}\cdot b^{r_b}\in W_{1,M}$ the total exponent of $g$ is precisely $r_b$.
Therefore, $W_{N,M}$ corresponds to the triple $(M,W^0,1_\lamp)$ if $M$ is finite and to the triple $(0,W^0,1_\lamp)$ if $M=\infty$.
\end{proof}
\begin{corollary}~\label{NormalityOfH}
\begin{enumerate}
\item
For all $M\in\barN$ and $N\in\Nstar$, the subgroup $H_{N,M}$ is normal if and only if $N$ divides $M$.
In particular $H_{N,\infty}$ is normal for every $N$.
\item
For all $M\in\barN$ and $N\in\Nstar$, the subgroup $W_{N,M}$ is normal if and only if $N=1$.
\end{enumerate}
\end{corollary}
\begin{proof}
First, we prove that $H_{N,M}^0\acts b$ is always equal to $H_{N,M}^0$.
Indeed, for all $g=\prod_{t=1}^j (b^{i_t}cb^{-i_t})^{r_t}\in \A$ we have $g\acts b=\prod_{t=1}^j (b^{i_t-1}cb^{-i_t+1})^{r_t}$.
Hence, $g$ belongs to $H_{N,M}^0$ if and only if $g\acts b$ belongs to $H_{N,M}^0$.
We also trivially have that $1(b1b^{-1})^{-1}=1$ always belongs to $H_{N,M}^0$.
Therefore, $H_{N,M}$ is normal if and only if $g(g\acts b^M)^{-1}\in H^0$ for all $g\in \A$.

Suppose that $N$ does not divide $M$.
Take $g=c\in \A$.
Then $c(c\acts b^M)^{-1}=c(b^Mcb^{-M})^{-1}$.
The sum $\sum_{i_t\equiv 0}r_t=1\not\equiv 0\pmod k$ since $k\geq 2$.
Therefore $c(c\acts b^M)^{-1}\notin H_{N,M}^0$, which implies that $H_{N,M}$ is not normal.

On the other hand, suppose now that $N$ divides $M$. Then for all $g=\prod_{t=1}^j (b^{i_t}cb^{-i_t})^{r_t}\in \A$, we have
\[
g(g\acts b^M)^{-1}=\prod_{t=1}^j (b^{i_t}cb^{-i_t})^{r_t}\cdot\prod_{t=1}^j (b^{i_t-M}cb^{-i_t+M})^{-r_t}
\]
belongs to $H_{N,M}^0$.

Now, in the case of $W_{N,M}$, we have $(0\dots 0,0)\acts cbc^{-1}b^{-1}=(01{-1}1\dots \pm1,0)$.
Therefore, as soon as $N\geq 2$, $cbc^{-1}b^{-1}$ does not belongs to $W^0$ and $W_{N,M}$ is not normal.
For $N=1$, we have $(x,i)\acts b=(x,i+1)$ and $(x,i)\acts c=(x+1,i)$.
Therefore,
\[
	W_{1,M}=\left\{g=\prod_{t=1}^j (b^{i_t}cb^{-i_t})^{r_t}b^{r_b}\ \middle|\
	\begin{gathered}
	r_b\equiv 0\pmod M\\
	\sum_{t=1}^j r_t\equiv 0\pmod k
	\end{gathered}
	\right\}
\]
is a normal subgroup.
\end{proof}
In particular, this implies that for $N>1$ dividing $M$, the subgroups $H_{N,M}$ and $W_{N,M}$ are not conjugate (equivalently the non-oriented Schreier graphs are not strongly isomorphic).
On the other hand, an easy check shows that $H_{N,1}=W_{N,1}$.
A careful look at the order of the image of $c$ in the corresponding action on the vertex set of $\orSw{k}{N}{M}$ shows that for $M<\infty$ and $N>1$, $H_{N,M}$ and $W_{N,M}$ are nearly never conjugate (this can happen only if $k$ divides all the $\binom {\lcm(M,N)}j$ for $1\leq j\leq N-1$).

Proposition~\ref{PropHW} (Part 1) and Corollary~\ref{NormalityOfH} imply the following.
\begin{theorem}\label{thmSpiderasfiniteCayley}
For $M=Nl$, the spider-web graph $\orSw{k}{N}{Nl}$ is weakly isomorphic to the Cayley graph of
\[
	\bigl(\oplus_{i=1}^N\Z/k\Z\bigr)\rtimes \Z/Nl\Z
\]
where the action of $\Z/Nl\Z$ on $\oplus_{i=1}^N\Z/k\Z$ is by shift, and with respect to the generating set $\{bc^i\}_{i=0}^{k-1}$ where $b$ is a generator of $\Z/Nl\Z$ and $c$ a generator of $\Z/k\Z$.

In particular, if we write $\lamp_{k,N}=\Z/k\Z\wr\Z/N\Z$ for the finite lamplighter group, we have 
for $N$ and $l$ coprime,
\[
	\unkSw{k}{N}{Nl}\iso\unCayley{\lamp_{k,N}\times \bigl(\Z/l\Z\bigr)}{\{(\bar c_r,1)\}_{r=0}^{k-1}}\]
which, for $l=1$, gives
\[
	\unkSw{k}{M}{M}\iso\unCayley{\lamp_{k,M}}{X}
\]
and for $N=1$
\[
	\unkSw{k}{1}{M}\iso\unCayley{\Z/k\Z\times \Z/M\Z}{\{(r,1)\}_{r=0}^{k-1}}.
\]
We also have
\[
	\unkSw{k}{1}{\infty}\iso\unCayley{\Z/k\Z\times \Z}{\{(r,1)\}_{r=0}^{k-1}}.
\]
\end{theorem}
\begin{proof}
The graph $\orSw{k}{N}{Nl}$ is weakly isomorphic to the graph $\orSchrei{\lamp}{H_{N,Nl}}{X}$.
Since $H_{N,Nl}$ is normal, this graph is strongly isomorphic to the Cayley graph of $G=\lamp/H_{N,Nl}$.
We know that $G=\langle b,c\rangle$ and that $c^k=1$ in $G$.
Moreover, $b^{Nl}$ and $cb^Nc^{-1}b^{-N}$ belong to $H_{N,Nl}$ and thus in $G$ the relations $b^{Nl}=1$ and $cb^N=b^Nc$ are true.
Therefore, $G$ is a quotient of
\[
	L=\presentation{b,c}{c^k,cb^Nc^{-1}b^{-N},b^{Nl},[c,b^jcb^{-j}];j\in\N}\iso\bigl(\oplus_{i=1}^N\Z/k\Z\bigr)\rtimes \Z/Nl\Z
\]
where the action of $\Z/Nl\Z$ is by shift.
We have $\abs L=k^N\cdot N\cdot l$ and $\abs G=k^N\cdot Nl$ (the number of vertices of $\orSw{k}{N}{Nl}$).
This implies $G=L$.

If $l$ and $N$ are coprime, $\Z/Nl\Z\iso\Z/N\Z\times \Z/l\Z$ acts on $\oplus_{i=1}^N\Z/k\Z$, with $\Z/N\Z$ acting by shift and $\Z/l\Z$ acting trivially.
Hence, $L\iso\lamp_{k,N}\times\Z/l\Z$.

Generating sets are images of $\{\bar c_r\}$ in $L$.

Finally, $\orSw{k}{1}{\infty}$ and $\orCayley{\Z/k\Z\times \Z}{\{(-r,-1)\}_{r=0}^{k-1}}$ have the same vertex set: $\{1,\dots, k\}\times\Z$.
Moreover, in both graphs, there is an edge from $(i,s)$ to $(j,t)$ if and only if $t=s+1$.
\end{proof}
\begin{remark}It is interesting to observe that the family of spider-web graphs $\orkSw{k}{N}{M}$ interpolates between Cayley graphs of direct products of finite cyclic groups and Cayley graphs of wreath products of finite cyclic groups, with the corresponding generating sets.
\end{remark}
Observe however that more graphs $\unkSw{k}{N}{M}$ can a priori be weakly isomorphic to Cayley graphs of some finite groups than those given in Theorem~\ref{thmSpiderasfiniteCayley}.
For example, one can check by hand that this is the case of $\unkSw{2}{2}{3}$, thought $H_{2,2,3}$ is not normal.
%
%
%
%
%
\subsection{Transitivity}
We now investigate the vertex-transitivity of spider-web graphs. 
We already know from the last subsection that if $N$ divides $M$ the spider-web graphs are weakly isomorphic to Cayley graphs and therefore are weakly transitive.
We will give a complete characterization of transitivity for spider-web graphs, but before that we need a technical lemma.
\begin{lemma}
For any $d\in\Nstar$ and any vertices $v$ and $w$ in $\unSw{k}{N}{M}$, there is a bijection between closed reduced paths of derangement $0$ of length $d$ with initial vertex $v$ and the ones with initial vertex $w$.
\end{lemma}
\begin{proof}
The bijection is easily seen on $\unSchrei{\lamp}{H_{N,M}}{X}$.
Recall that we have
\[
	(x_1\dots x_N,i)\acts \bar c_r=\bigl((x_N-r)x_1\dots x_{N-1},i-1\bigr).
\]
For any reduced path $p$ with initial vertex $(x_1\dots x_N,i)$, there is a unique path $q$ with initial vertex $(y_1\dots y_N,j)$ with the same label $l\in\lamp$.
Since the derangement is $0$, we have that
\[
	(z_1\dots z_N,i)\acts l=((z_1+a_1)\dots (z_N+a_N),i)
\]
for some integers $a_i$.
In particular, the end vertex of $p$ is $((x_1+a_1)\dots (x_N+a_N),i)$, while the end vertex of $q$ is $((y_1+a_1)\dots (y_N+a_N),j)$.
Therefore, $p$ is closed if and only if all the $a_i$'s are equal to $0$, if and only if $q$ is closed.
\end{proof}

\begin{theorem}\label{thmTransitivity}
For all $M\in\barN$ and $N\in\Nstar$, the graphs $\orSw{k}{N}{M}$ and $\unSw{k}{N}{M}$ are weakly transitive (i.e. is transitive by weak automorphisms) if and only if $M\geq N$.
\end{theorem}
\begin{proof}
First, observe that any automorphism of $\orSw{k}{N}{M}$ naturally induces an automorphism of $\unSw{k}{N}{M}$.
Therefore, it is sufficient to prove that if $M\geq N$ then $\orSw{k}{N}{M}$ is weakly transitive and that if $M<N$, the graph $\unSw{k}{N}{M}$ is not transitive.

It is easy to check that for every $M$ the function $\eta$ on $\vec C_M$ defined by $\eta(i)\defegal i+1$ (where the addition is taken modulo $M$) is an automorphism.
Therefore, $T\defegal\Id\otimes\eta$ is an automorphism of $\orSw{k}{N}{1}\otimes\vec C_M\iso\orSw{k}{N}{M}$.
It is even a strong automorphism for every $M$ --- even for $M$ smaller than $N$~--- since the labeling of $\orSw{k}{N}{M}$ comes from the labeling of $\orSw{k}{N}{1}$ and the fact that $\Id$ is a strong automorphism.

We now define another function $\psi$ on $\orSw{k}{N}{M}$ by the following formula on vertices:
\[
	\psi(x_1\dots x_N,t)\defegal
		\begin{cases}
		(x_1\dots(x_{N-t}+1)\dots x_N,t) & \tn{ if } 0\leq t\leq N-1\\
		(x_1\dots x_N,t)\tn{ else. }
		\end{cases}
\]
We define $\psi$ on edges in the following way: the unique edge with initial vertex $(x,t)$ and label $i$ is sent on the unique edge with initial vertex $\psi(x,t)$ and label $i$ if $t\not\equiv -1\pmod M$, or on the edge with label $i+1$ if $t\equiv -1\pmod M$.
We claim that with this definition, $\psi$ is a weak isomorphism if $M\geq N$.
To prove that, it remains to check that for any edge $e$, $\tau\bigl(\psi(e)\bigr)=\psi\bigl(\tau(e)\bigr)$.
Since the definition of $\psi$ depends on $t$, we have four different cases.
The first is when $0\leq t\leq N-2$.
The second is for $t=N-1$.
The third when $N-1<t<M-1$ and the last one when $t=M-1$.
The first, second and fourth cases are easy computations left to the reader.
In the third case, $\psi$ acts as the identity and there is nothing to prove.
We have
\begin{align*}
	T^{-i}\psi^{x_i}T^i(0\dots 0,0)&=T^{-i}\psi^{x_i}(0\dots 0,i)\\
	&=(0\dots x_i\dots 0,i)\\
	&=(0\dots x_i\dots 0,0).
\end{align*}
Therefore, for any vertex $(x_1\dots x_N,t)$ we have
\[
	\biggl(T^i\cdot\prod_{j=0}^n T^{-j}\psi^{x_j}T^j\biggr)  (0\dots 0,0)=(x_1\dots x_N,t),
\]
which proves the transitivity of $\orSw{k}{N}{M}$ when $M\geq N$.

Now, if $M<N$ look at the two vertices
\[
	v=(0\dots 0,0)\quad\textnormal{and}\quad w= (10\dots 0,0)
\]
of $\unSw{k}{N}{M}$.
By the previous lemma, there is the same number of closed paths with derangement $0$ and length $M$ based at $v$ and at $w$.
Since $M<N$, there is no closed path of length $M$ starting at $w$ with non-zero derangement.
On the other hand, there is at least one such closed path starting at $v$: the path where all edges have label $R_0$.
Therefore, there is strictly more closed paths of length $M$ starting at $v$ than such paths starting at $w$.
This implies that $\unSw{k}{N}{M}$ is not transitive.
\end{proof}
\begin{remark}
It is possible to demonstrate a refinement of this theorem. Namely that if $M<N$, the number of orbits of $\orkSw{k}{N}{M}$ under its group of automorphisms is bounded from below by $\frac NM$ and from above by $k^{N-M}$, and the number of orbits of $\unkSw{k}{N}{M}$ under its group of automorphisms is bounded from below by $\max(2,\frac N{2M})$ and from above by $k^{N-M}$.
In particular, for $k$ and $M$ fixed, the number of orbits of $\unkSw{k}{N}{M}$ is unbounded.
\end{remark}
\subsection{Line graphs, Eulerian and Hamiltonian cycles and coverings}
The family of de Bruijn graphs is well-known to enjoy some nice graph-theoretic properties.
The aim of this subsection is to verify that the family of spider-web graphs share many of them and can thus be indeed viewed as a natural extension of de Bruijn graphs.
\begin{proposition}
For all $M\in\barN$ and $N\in\Nstar$, the spider-web graph $\orSw{k}{N+1}{M}$ is (weakly) isomorphic to the line graph of $\orSw{k}{N}{M}$.
\end{proposition}
\begin{proof}
This follows from the same result for de Bruijn graphs (Lemma~\ref{LemmaLineDeBruijn}), the fact that $\orSw{k}{N}{M}$ is the tensor product $\orB_{N}\otimes \vec C_M$, Lemma~\ref{PropLineTensor} and the fact that $\vec C_M$ is its own line graph.
\end{proof}
\begin{proposition}
For all $M\in\barN$ and $N\in\Nstar$, the spider-web graph $\orSw{k}{N}{M}$ is Eulerian (there exists a  closed path $p$ consisting of edges of $\orSw{k}{N}{M}$ that visits each edge exactly once) and Hamiltonian (there exists a  closed path that visits each vertex exactly once)
\end{proposition}
\begin{proof}
The directed graph $\orSw{k}{N}{M}$ is finite, connected and for every vertex $v$ in $\orSw{k}{N}{M}$ the number of outgoing edges is equal to the number of ingoing edges. Therefore, $\orSw{k}{N}{M}$ is Eulerian.

For $N\geq 1$, the graph $\orSw{k}{N}{M}$ is isomorphic to the line graph  of $\orSw{k}{N-1}{M}$. This line graph is Hamiltonian since $\orSw{k}{N-1}{M}$ is Eulerian.
Finally, $\orSw{k}{0}{M}$ is a ``thick'' oriented circle: the vertex set is $M$ and for every vertex $i$ there is $k$ edges from $i$ to $i+1$. This graph is obviously Hamiltonian.
\end{proof}
We proved that $\orSw{k}{N}{M}$ is Eulerian and Hamiltonian as an oriented graph.
That is, the closed path in question consists only of edges of $\orSw{k}{N}{M}$.
This trivially implies that $\unSw{k}{N}{M}$ is Eulerian and Hamiltonian.
Indeed, for an oriented graph $\vec\Theta$, being Eulerian (or Hamiltonian) as an oriented graph is a stronger property that $\underline{\vec\Theta}$ being Eulerian (or Hamiltonian) as a non-oriented graph.
Finally, we generalize Corollary~\ref{corCovering} and show that spider-web graphs form towers of graphs coverings, both in $N$ and, in a certain sense, in $M$.
\begin{proposition}
For all $M\in\barN$ and $N\in\N$, the oriented graph $\orSw{k}{N}{M}$ (weakly) covers $\orSw{k}{N-1}{M}$.

For every $i\in\N$, the oriented graph $\orSw{k}{N}{iM}$ (weakly) covers $\orSw{k}{N}{M}$.
\end{proposition}
\begin{proof}
By Corollary~\ref{corCovering}, we know that $\orB_{N+1}$ covers $\orB_N$ and it is easily seen that $\vec C_{iM}$ covers $\vec C_M$.
A simple application of Lemma~\ref{lemmaCovering} gives the desired result. 
\end{proof}
Note that any covering of oriented graphs $\phi\colon\vec\Delta_1\to\vec\Delta_2$ naturally induces a covering between the underlying graphs.

\end{document}